\documentclass{amsart}
\usepackage{amsmath,amscd,amssymb,amsthm}
\usepackage{amsfonts}
\usepackage{mathabx}
\usepackage[all]{xy}
\usepackage{mathtools}
\usepackage{enumerate}

\usepackage{enumitem}

\usepackage{hyperref}

\usepackage{adjustbox}

\hypersetup{
  pdftitle={Weak and genuine An-formality for number fields},
  pdfauthor={Ambrus Pál and Gereon Quick}
}

\usepackage[usenames, dvipsnames]{color}

\definecolor{NTNUblue}{RGB}{0,80,158}
\definecolor{NTNUbluesupport}{RGB}{62,98,138}
\definecolor{NTNUorange}{RGB}{239,129,20}

\usepackage[OT2,T1]{fontenc}
\DeclareSymbolFont{cyrletters}{OT2}{wncyr}{m}{n}
\DeclareMathSymbol{\Sha}{\mathalpha}{cyrletters}{"58}

\newcommand{\Fp}{\mathbb F_p}

\newcommand{\HH}{\operatorname{HH}}
\newcommand{\Frob}{\operatorname{Frob}}
\newcommand{\Gal}{\operatorname{Gal}}
\newcommand{\Gov}{\operatorname{Gov}}
\newcommand{\Nm}{\operatorname{N}}
\newcommand{\Aut}{\operatorname{Aut}}
\newcommand{\Hom}{\operatorname{Hom}}

\newcommand{\im}{\operatorname{im}}
\newcommand{\coker}{\operatorname{coker}}
\newcommand{\res}{\operatorname{res}}

\newcommand{\angles}[1]{\left\langle #1\right\rangle}

\newcommand{\loc}{\operatorname{loc}}

\newcommand{\Un}{{\mathbb U}_n}

\newcommand{\cd}{\operatorname{cd}}

\newcommand{\ob}{\mathfrak o}

\newcommand{\qq}{\mathfrak q}

\newcommand{\Br}{\operatorname{Br}}
\newcommand{\Sol}{\operatorname{Sol}}

\newcommand{\uA}{\mathfrak u}

\newcommand{\UA}{\mathbb U}
\newcommand{\Krel}{K}

\newcommand{\MD}{M'}

\newcommand{\into}{\hookrightarrow}

\newcommand{\xto}{\xrightarrow}

\newcommand{\F}{\mathbb{F}}

\newcommand{\Q}{\mathbb{Q}}
\newcommand{\R}{\mathbb{R}}
\newcommand{\Z}{\mathbb{Z}}

\newcommand{\Ch}{\mathcal{C}}

\newcommand{\Fh}{\mathcal{F}}

\newcommand{\Zh}{\mathcal{Z}}

\newcommand{\ann}{\mathrm{an}}

\newcommand{\gHom}{\underline{\Hom}}

\newcommand{\cont}{\mathrm{cont}}

\newcommand{\id}{\mathrm{id}}

\newcommand{\bbb}{\bullet}

\newcommand{\Cb}{\Ch^\bbb}

\newcommand{\Zb}{\Zh^\bbb}

\newcommand{\Hb}{H^\bbb}

\newcommand{\dee}{\partial}

\newcommand{\One}{\mathbf{1}}

\newcommand{\ot}{\otimes}

\newcommand{\rr}{\rho} 

\newcommand{\brr}{\overline{\rho}}

\theoremstyle{plain}

\newtheorem{theorem}{Theorem}[section]

\newtheorem{proposition}[theorem]{Proposition}
\newtheorem{lemma}[theorem]{Lemma}

\newtheorem{corollary}[theorem]{Corollary}

\theoremstyle{definition}
\newtheorem{definition}[theorem]{Definition}

\newtheorem{example}[theorem]{Example}

\newtheorem{notn}[theorem]{Notation}

\theoremstyle{remark}

\newtheorem{remark}[theorem]{Remark}

\title{Weak and genuine $A_{\lowercase{n}}$-formality for number fields}

\author{Ambrus P\'al}
\address{Mathematical Institute, E\"otv\"os Lor\'and University, H-1117 Budapest, Hungary; 
\newline \hspace*{3.2mm} 
Department of Mathematical Sciences, UAEU,  UAE}
\email{ambrus.pal@uaeu.ac.ae}
\author{Gereon Quick}
\address{Department of Mathematical Sciences, NTNU, NO-7491 Trondheim, Norway}
\email{gereon.quick@ntnu.no}
\thanks{Both authors were partially supported by %RCN Project No.\,313472 {\it Equations in Motivic Homotopy}, and   
the project \emph{Pure Mathematics in Norway} funded by the Trond Mohn Foundation.}


\begin{document}

\begin{abstract}
We introduce weak $A_n$-formality for Galois cochain algebras 
and, for profinite groups whose mod-$p$ cohomology is concentrated in degrees at most two, 
characterise weak and genuine $A_n$-formality 
through finite embedding problems and the projective system they form. 
We then study $A_n$-formality for number fields. 
Let $K$ be a number field and let $p$ be an odd prime
with $\mu_p\not\subset K$. For $p\geq 5$, and for $p=3$ under a natural
local cyclotomic condition, we prove that every associated finite 
$A_3$-embedding problem is solvable. 
In contrast, for every odd $p$ with $\mu_p\not\subset K$, 
we show that no compatible family of finite-stage 
solutions exists, and hence $C^\bullet(G_K,\mathbb F_p)$ is not
genuinely $A_3$-formal. 
The proof uses Chebotarev density, governing fields, and the Gras--Munnier criterion
to obstruct the compatibility of finite-stage solutions, 
and it builds on previous work of Maire--Min\'a\v c--Ramakrishna--T\^an. 
This shows that genuine $A_3$-formality detects a global compatibility obstruction for number fields 
invisible to every individual finite embedding problem and 
invisible to obstructions arising from individual strong Massey vanishing problems. 
\end{abstract}

\maketitle
%\end{nouppercase}

%%%%%%%%%%

\section{Introduction}

Let $F$ be a field, let $G_F$ be its absolute Galois group, and let $p$ be a prime. 
The differential graded algebra   $\Cb(G_F,\F_p)$ 
of continuous inhomogeneous cochains 
contains more information than its cohomology ring $\Hb(G_F,\F_p)$. 
Formality asks whether this extra information can be discarded, up to quasi-isomorphism. 
This question is closely related to the behaviour of Massey products.
In \cite{HW}, Hopkins and Wickelgren proved that all defined triple Massey products
over local and global fields vanish at the prime $p=2$, 
and asked whether this phenomenon reflects formality of the Galois cochain
algebra. 
While Positselski \cite{PoMoscow, Po} showed 
that $\Cb(G_F,\F_2)$ is not formal in general for  local fields, 
Min\'a\v c and T\^an conjectured in \cite[Conjecture 1.6]{MT2} that 
the computation of Hopkins--Wickelgren generalises and all fields satisfy $n$-Massey vanishing with respect to $p$, 
i.e., all $n$-fold Massey products of elements in $H^1(G_F,\F_p)$ vanish whenever they are defined. 
In fact, by the work of % Matzri \cite{Matzri}, 
Efrat--Matzri \cite{EM17} and Min\'a\v c--T\^an \cite{MT1}, 
all fields satisfy triple Massey vanishing with respect to all primes. 
For a more recent, independent proof, see \cite{Lam_etal}. %\cite[Section 6.2]{Lam_etal}. 
In \cite{HarpazWittenberg}, Harpaz--Wittenberg showed that number fields satisfy $n$-Massey vanishing with respect to all primes. 
More recently, Merkurjev--Scavia proved in \cite{MS2} that all fields satisfy fourfold Massey vanishing with respect to $p=2$. 
Other cases of the conjecture have been proven in \cite{MTE}, \cite{PQ}, \cite{PSz}, and \cite{Quadrelli}. 
Maire--Min\'a\v c--Ramakrishna--T\^an 
\cite{MMRT} show that number fields not containing all $p$-th roots of unity % checked and updated from arXiv-version 
even satisfy a strong form of $n$-Massey vanishing. 
However, Galois cochain algebras need not be formal at $p=2$. 
Examples of Harpaz and
Wittenberg \cite[Example A.15]{GMT}, 
and later Merkurjev and Scavia in \cite{MS1,MSNF} % \cite[Theorem 1.6]{MS1} and \cite[Theorem 1.3]{MS2023}, 
show that fourfold Massey products need not even be defined when all neighbouring cup
products vanish.  

Formality of a differential graded algebra $\Cb$ is equivalent to the vanishing of the higher operations 
of its transferred minimal \(A_\infty\)-structure on its cohomology algebra $\Hb$. 
This allows one to study formality one operation at a time, 
and  $\Cb$ is $A_n$-formal if the transferred minimal structure can be chosen so that the higher operations vanish up to level $n$. 
In this paper, we propose to study a weaker version of $A_n$-formality, 
called {\em weak $A_n$-formality}, 
which requires the vanishing of all higher operations up to level $n$ 
after restriction to every finite-dimensional vector subspace $V \subset H^1$, 
see Section \ref{sec:weak_vs_genuine_An_general} for a precise definition. 
Genuine $A_n$-formality implies weak $A_n$-formality. 
Under the assumption $H^i=0$ for $i \ge 3$, the converse holds 
if the trivialisations on each finite-dimensional subspace $V$ can be chosen compatibly. 
The main purpose of this paper is to show that this compatibility
problem provides %is a genuine 
a new arithmetic obstruction for number fields. % for $n=3$. 
%
%%%%%%%%%

\subsection{Main results} 

Our first main result is a reformulation of $A_n$-formality for the cochain algebra $\Cb(G,F_p)$ of a profinite group $G$ 
in terms of the existence of compatible solutions of certain embedding problems for $G$. 
More concretely, for every $n\ge 2$ and every vector subspace $V \subseteq H^1(G,\F_p)$, 
we construct a profinite group $\UA_n(V)$ arising as the Pontryagin dual of a construction 
similar to the Koszul complex of a quadratic algebra; see Section \ref{sec:finite-embedding-comparison}. 
We then obtain an embedding problem 
$\mathbf E^{(n)}_V \colon G\xrightarrow{\chi_V} V^\vee \twoheadleftarrow\UA_n(V)$ 
of profinite groups where %$V^\vee$ is the Pontryagin dual and 
$\chi_V$ is a canonically defined character. 
Our main technical result on $A_n$-formality is the following characterisation in terms of projective systems of finite embedding problems, 
which may be of independent interest, 
see Theorem \ref{thm:compatible-defining-systems}. 
\begin{theorem}
\label{thm:compatible-defining-systems_intro}
Let $n \ge 3$, and let $G$ be a profinite group such that $H^i(G,\F_p) = 0$ for $i\ge3$.  
Then $\Cb(G,\F_p)$ is $A_n$-formal 
if and only if 
all embedding problems $\mathbf E^{(n)}_V$ for finite-dimensional $V \subseteq H^1(G,\F_p)$ 
admit a compatible family of solutions. 
\end{theorem}

We use Theorem \ref{thm:compatible-defining-systems_intro} 
to prove the following local--global principle for number fields, 
where $\mu_q$ denotes the group of $q$-th roots of unity, 
see Theorem \ref{thm:LG-equivalence}. 

\begin{theorem}
\label{thm:LG-equivalence_intro}
Let $K$ be a number field, let $p$ be odd, assume
$\mu_p\not\subset K$, 
and let $n\ge3$.  
Then the following are equivalent: 
\begin{enumerate}[label=\textnormal{(\roman*)}]
\item $\Cb(G_K,\Fp)$ is weakly $A_n$-formal in degree one, 
\item for every place $v$ of $K$, % the local cochain dga
$\Cb(G_{K_v},\Fp)$ is $A_n$-formal.
\end{enumerate}
\end{theorem}

%%%%%%%%%%%%%

We apply the local--global principle to the case $n=3$, 
where our complete computation of the local obstruction class applies \cite{PQodd}. 
Weak $A_3$-formality already implies the cup-defining fourfold Massey product property, 
i.e., the fourfold Massey product of classes in $H^1$ is defined 
whenever the neighbouring cup-products vanish. 
For a number field $K$ and an odd prime $p$ such that $\mu_p\not\subset K$, 
Maire--Min\'a\v c--Ramakrishna--T\^an show in \cite{MMRT} 
that $\Cb(G_K,\F_p)$ satisfies strong Massey vanishing. 
Hence all obstructions arising from individual strong Massey problems vanish. 
The main results of this paper, however, 
demonstrate that the distinction between weak and genuine formality 
is even more subtle for number fields % at odd primes.  
as we will now describe. 
Our first key result is the following, see Theorem \ref{thm:numberfield-absolute-trim}. 

\begin{theorem}
\label{thm:intro-weak}
Let $K$ be a number field, let $p$ be an odd prime number, and assume $\mu_p\not\subset K$.  
For $p=3$, we assume that, for every place $v$ of $K$, 
if  $\mu_3 \subset K_v$ then $\mu_9 \subset K_v$. 
Then $\Cb(G_K,\F_p)$ is weakly $A_3$-formal.  
\end{theorem}

We provide two proofs of this result.  
Both proofs are based on the local $A_3$-formality theorem for odd-prime Demushkin groups, % of the authors \cite{PQodd}, 
and the extra hypothesis at $p=3$ excludes the unique local Demushkin case in which the $A_3$-obstruction is nonzero \cite[Theorem 1.2]{PQodd}.
The first proof then uses the characterisation of $A_3$-formality via embedding problems and then 
applies the general $p$-group lifting theorem of Maire--Min\'a\v c--Ramakrishna--T\^an, \cite[Theorem 3.2]{MMRT} 
to suitable local solutions. 
The second proof uses the P\'al--Schlank Hasse principle for finite embedding problems with abelian kernel \cite{PS}, 
where the cyclotomic assumption $\mu_p\not\subset K$ implies that the relevant Tate--Shafarevich obstruction vanishes. 
Our main result, however, is that, for number fields,  
weak $A_3$-formality is strictly weaker than genuine $A_3$-formality,  
see Theorem \ref{thm:numberfield-nonformal}. 

\begin{theorem}
\label{thm:intro-main}
Let $K$ be a number field, let $p$ be an odd prime number, and assume $\mu_p\not\subset K$.  
Then $\Cb(G_K,\F_p)$ is not $A_3$-formal.
\end{theorem}

%%%%%%%%%%%%%%%%%%

Before we give an outline of the arguments to prove  
Theorem \ref{thm:intro-main}, we discuss the remaining cases. 
%
%%%%%%%%%%%%%
%
%
At $p=3$, weak $A_3$-formality depends on the local cyclotomic
structure. 
Although strong Massey vanishing holds whenever
$\mu_3\not\subset K$ by the work of Maire--Min\'a\v c--Ramakrishna--T\^an
\cite[Theorem 1.1]{MMRT},  
the computation of the obstruction class of pro-$3$ Demushkin groups in \cite{PQodd} and 
our local-global principle imply that weak $A_3$-formality fails 
if there is a place $v$ such that $\mu_3\subset K_v$ but $\mu_9\not\subset K_v$, 
see Theorem \ref{thm:numberfield-absolute-trim_fails_for_p=3}. 

\begin{theorem}
\label{thm:numberfield-absolute-trim_fails_for_p=3_intro}
Let $K$ be a number field, and assume $\mu_3 \not\subset K$. 
Assume that  $K$ has a place $v$ such that $\mu_3 \subset K_v$ 
but $\mu_9 \not\subset K_v$. 
Then $\Cb(G_K,\F_3)$ is not weakly $A_3$-formal. 
In particular, $\Cb(G_{\Q},\F_3)$ is not weakly $A_3$-formal. 
%In particular, $\Cb(G_K,\F_3)$ is not $A_3$-formal. 
\end{theorem}

For $p=2$, the example of Harpaz--Wittenberg in \cite[Example A.15]{GMT}  
(also recovered by Merkurjev--Scavia in \cite[Proposition 5.3]{MS1})  
shows that strong Massey vanishing does not hold for $\Q$, 
and hence $\Cb(G_{\Q},\F_2)$ is not weakly $A_3$-formal. 
We provide further examples of the failure of weak $A_3$-formality for $\Cb(G_F,\F_2)$ where $F$ is a field 
in Proposition \ref{prop:nonA3-examples}. 
In particular, we show that $\Cb(G_F,\F_2)$ is not weakly $A_3$-formal  
whenever $F$ is a finite extension of $\Q$ of odd degree. 
In Remarks \ref{rem:failure_at_p=2_local-global} and \ref{rem:failure_at_p=2_weak_A_3}, 
we explain why neither the local--global principle of Theorem \ref{thm:LG-equivalence_intro} 
nor the second proof of Theorem \ref{thm:intro-weak} extend to $p=2$. 
The above examples show that % in Proposition \ref{prop:nonA3-examples} show that
this is not merely a limitation of the proofs as weak $A_3$-formality itself can fail at $p=2$. 
%
%%%%%%%%%%%%%%%

\subsection{Outline of the proof of Theorem \ref{thm:intro-main}}

By Theorem \ref{thm:compatible-defining-systems_intro}, 
$A_3$-formality would imply the existence of a compatible
family of solutions to the $A_3$-embedding problems 
$\mathbf E^{(3)}_V \colon G_K \xrightarrow{\chi_V} V^\vee \twoheadleftarrow\UA_3(V)$ 
attached to all finite-dimensional subspaces $V$ of $H^1(G_K,\F_p)$. 
We begin with a two-dimensional subspace 
$V=\F_p a + \F_p b$
generated by linearly independent classes satisfying $a\cup b=0$, and
consider the $V$-component $\rho_V \colon G_K \to \UA_3(V)$ of such a family. 
Let $\bar\rho_V \colon G_K \to \UA_2(V)$ denote its quadratic quotient. 
If $\bar\rho_V$ is trivial on the kernel of the tautological character $\chi_V$, then it factors through
$V^\vee$ and produces a section of $\UA_2(V)\twoheadrightarrow V^\vee$. 
An elementary commutator calculation shows
that no such section can exist when $V$ contains the isotropic pair $a,b$.
%
%%%%%
%
It remains to treat the case in which $\bar\rho_V$ is nontrivial on $\ker(\chi_V)$. 
A nonzero quadratic coordinate of $\bar\rho_V$ is detected
at a new tame prime $v$ of $K$, 
which we can choose by Chebotarev's density theorem. 
Cyclotomic disjointness allows us simultaneously to require that $v$ splits in the
cyclotomic and governing fields. 
The one-prime Gras--Munnier criterion of Corollary \ref{cor:gras-one-prime} 
then produces a new class $c\in H^1(G_K,\F_p)$  
which is ramified only at $v$, trivial at the previous primes, 
and with trivial cup-products with elements in $V$. 
For the enlargement $W=V+\F_p c$, 
the existence of a lift $\rho_W$ %degree-three lifting equation 
would force the cup product of a nonzero unramified local class with the ramified class $c_v$
to vanish. 
This contradicts the nondegeneracy of the local Tate pairing. 
Thus $\rho_V$ cannot extend to the problem attached to $W$. 
Consequently, no family of finite-stage solutions can be compatible under all
enlargements, and hence the Galois cochain algebra is not $A_3$-formal.

%%%%%%%%%%%%%%%%%

\begin{remark}\label{rem:MMRT-relation}
The ideas in this paper are inspired by and often rely on the work of Maire--Min\'a\v c--Ramakrishna--T\^an in \cite{MMRT}. 
The proof of the local--global principle in Theorem \ref{thm:LG-equivalence_intro} 
and the arithmetic mechanism of the proof of Theorem \ref{thm:intro-main} rely on the results proven in \cite{MMRT}. 
To prove Theorem \ref{thm:LG-equivalence_intro}, 
we use \cite[Theorem 3.2]{MMRT}, which is a general lifting theorem for finite $p$-groups with given local lifts. 
To prove Theorem \ref{thm:intro-main}, we use, in particular, 
the theory of governing fields and the cyclotomic disjointness argument of \cite[Section 2]{MMRT}, 
together with the Gras--Munnier criterion. 
The resulting construction is closely related to the method underlying \cite[Proposition 3.1]{MMRT}. 
However, we use it in the opposite direction. 
While Maire--Min\'a\v c--Ramakrishna--T\^an use their theory to show the solvability of strong Massey vanishing problems, 
we apply the techniques 
to choose a finite-dimensional subspace $V$ such that every solution for $V$ is obstructed after some finite-dimensional enlargement, 
so that no compatible solution exists in the whole projective system of embedding problems. 
\end{remark}

%%%%%%%%%%%%%

\subsection{Content} 
In Section \ref{sec:weak_vs_genuine_An_general}, 
we define weak $A_n$-formality and develop a criterion for weak and genuine $A_n$-formality under the assumption 
that the cohomology is concentrated in degrees at most two. 
In Section \ref{sec:finite-embedding-comparison}, 
we  construct the groups $\UA_n(V)$, identify the restricted $A_n$-obstructions with embedding problems, 
and characterise genuine $A_n$-formality by the existence of a compatible inverse system of solutions 
proving Theorem \ref{thm:compatible-defining-systems_intro}. 
In Section \ref{sec:weak-counterexamples}, we discuss explicit examples of the failure of weak $A_3$-formality at the prime $2$. 
In Section \ref{sec:numberfields-local-global}, 
we prove the local--global principle for weak $A_n$-formality over number fields at odd primes of Theorem \ref{thm:LG-equivalence_intro}. 
In Section \ref{sec:number-fields-weak}, we apply this criterion and the computation for Demushkin groups to prove
Theorems \ref{thm:intro-weak} and \ref{thm:numberfield-absolute-trim_fails_for_p=3_intro}. 
Finally, in Section \ref{sec:nonformality}, we prove  Theorem \ref{thm:intro-main}. 

%%%%%%%%%%%%%%%%%%%%%

\subsection*{Acknowledgement} 
We are very grateful to Jan Min\'a\v c for his encouragement during the work on this project. 
The authors acknowledge that the initial sketches of some parts of the paper were developed via conversations with OpenAI's ChatGPT. 
The authors thoroughly reviewed the suggestions of the AI. 
The authors have written the entire manuscript and take full responsibility for its content. 

%%%%%%%%%%%%%%%%%%%%%
%
%%%%%%%%%%%%%%%%%%%%%

\section{Criteria for weak and genuine \texorpdfstring{$A_n$}{An}-formality}\label{sec:weak_vs_genuine_An_general}

Let $p$ be a prime number. 
Let $\Cb = (\Cb,\cup, \delta)$ be a strictly unital differential graded $\Fp$-algebra with cohomology algebra $\Hb$. 
We refer to \cite{Keller} for the theory of $A_\infty$-algebras. 
By Kadeishvili's work \cite{Kadeishvili82, Kadeishvili88}, 
homotopy transfer provides $\Hb$ with a strictly unital minimal $A_\infty$-structure
$(\Hb,m_2,m_3,m_4,\ldots)$ 
with $m_2$ the multiplication induced by the cup-product.  
The $n$-ary operation $m_n$ has degree $2-n$. 
We note that any two minimal transferred 
$A_\infty$-structures on $\Hb$ are $A_\infty$-isomorphic through an isomorphism whose linear term is the identity. 
We will therefore also refer to $(\Hb,m_2,m_3,m_4,\ldots)$ as {\em the minimal $A_\infty$-structure} on $\Hb$. 

\begin{definition}\label{def:An}
For $n\ge3$, 
$\Cb$ is called \emph{$A_n$-formal} if the transferred minimal $A_\infty$-structure 
can be chosen such that 
$m_3=m_4=\cdots=m_n=0$.
\end{definition}

%%%%%%%%%%%%%

\begin{example}
Every differential graded $\F_2$-algebra whose cohomology algebra is isomorphic to 
$\F_2[a,b,c]/(ab,bc)$ is $A_3$-formal by \cite[Theorem 1.1]{DQ}. 
\end{example}

%%%%%%%%%%%%%%%%%%%

\begin{remark}\label{rem:A_infty_identity_for_mn} 
If $m_3=\dots=m_{n-1}=0$, then the $A_\infty$-identity for the homomorphism  
$m_n \colon (\Hb)^{\otimes n} \to \Hb$ of degree $2-n$ is given by 
\begin{align*}%\label{eq:A_infty_identity_for_mn} 
0 = m_2(\One \otimes m_n) + (-1)^{n+1}m_2(m_n \otimes \One) 
+ \sum_{i=0}^{n-1} (-1)^{i+1} m_n\left(\One^{\otimes i} \otimes m_2 \otimes \One^{\otimes (n-1-i)} \right)
\end{align*}
where $\One$ denotes the identity on $\Hb$. 
\end{remark}

Once $m_3=m_4=\cdots=m_{n-1}=0$, 
the obstruction to finding a transferred minimal 
$A_\infty$-structure on $\Hb$ with trivial $m_n$ lies in the graded Hochschild cohomology group $\HH^{n,2-n}(\Hb,\Hb)$. 
We now recall the definition of graded Hochschild cohomology. 
Let $A = \oplus_{i \ge 0} A^i$ be a non-negatively graded $\Fp$-algebra with $A^0=\Fp$. 
Let $M$ be a graded $A$-bimodule. 
For $s\in \Z$, we write $M[s]$ for the graded $A$-bimodule given in degree $n$ by $M[s]^n = M^{n+s}$, 
where $A$ acts on the left by $a\cdot_{M[s]}m = (-1)^{s|a|}am$ and from the right by $m \cdot_{M[s]} a = ma$. 
We write $A_+ \coloneqq \oplus_{i \ge 1} A^i$ %A/\F_p\langle 1 \rangle$ 
for the augmentation ideal. 
We let $\gHom_{\Fp}(A_+, M)$ denote the graded $\Fp$-vector space
with degree $s$ component given by $\Fp$-linear graded maps $f \colon A_+ \to M[s]$. 
The {\it normalised Hochschild complex} of $A$ with coeﬃcients in $M$,  
denoted by $C^{\bullet,*}(A,M)$, is defined as follows. 
We set 
\begin{align*}%\label{eq:def_graded_complex_HH}
C^{n,s}(A,M) \coloneqq  \gHom^s_{\Fp}(A_+^{\ot n}, M) 
\end{align*}
with differential $\dee^{n-1} \colon C^{n-1,s}(A,M) \to C^{n,s}(A,M)$ 
defined by 
\begin{align}\label{eq:HH_differential}
\dee^{n-1}(f)(a_1, \ldots, a_{n}) \coloneqq & ~  (-1)^{s|a_1|}a_1f(a_2, \ldots, a_{n}) \\
& ~~~ + \sum_{i=1}^{n-1} (-1)^i f(a_1, \ldots, a_ia_{i+1}, \ldots, a_{n})  \nonumber \\
&  ~~~ + (-1)^{n} f(a_1, \ldots, a_{n-1})a_{n} \nonumber
\end{align}
where $|a_1|$ denotes the degree of the element $a_1$. 
A Hochschild cochain in degree $(n,s)$ with coefficients in $M$ is thus given by a multilinear function 
from $n$-tuples of elements of $A$ to $M$ which raise the degree by $s$. 
We note that the differential only changes the first grading denoted by $\bbb$. 
The {\it Hochschild cohomology group} $\HH^{n,s}(A,M)$ is the $n$th cohomology group of the complex 
$(C^{\bbb,s}(A,M), \partial)$. 
The groups $\HH^{\bullet, *}(A,M)$ are equipped with a cohomological grading $\bullet$, 
and an internal grading $*$ induced by the grading of $A$ and $M$.  
%
%%%%%%%%%%%%%%
%

\begin{remark}
Evaluating the identity for $m_n$ in Remark \ref{rem:A_infty_identity_for_mn} 
on homogeneous elements and comparing with \eqref{eq:HH_differential} 
shows that $m_n$ is a Hochschild cocycle and defines an element $[m_n] \in \HH^{n,2-n}(\Hb,\Hb)$.  
\end{remark}

%%%%%%%%%%%%%%

We also refer to $A_n$-formality as {\em genuine} $A_n$-formality to distinguish from the following weaker notion. 

\begin{definition}\label{def:weak-An} 
We say that $\Cb$ is \emph{weakly $A_n$-formal} 
if, for every finite-dimensional vector subspace $V\subset H^1$, there is a choice of %transferred 
minimal structure  for which 
$m_3|_{V^{\otimes3}} = \cdots = m_n|_{V^{\otimes n}} = 0$.  
\end{definition}

\begin{remark}
We note that the definition of weak formality does not require any compatibility when $V$ varies. 
\end{remark}

%%%%%%%%%%%%%

From now on we assume that $\Cb$ is a differential graded $\Fp$-algebra 
such that $\Hb$ is concentrated in degrees $0,1,2$ with $H^0 = \F_p$. 
We note that this assumption holds for the mod-$p$ Galois cohomology of number fields when $p$ is odd.
Under this assumption, % $H^0 = \F_p$ and $H^i=0$ for $i \ge 3$,  
for every $n \ge 3$, 
the only possibly nonzero component of $m_n$ is  
$m_n \colon (H^1)^{\otimes n} \to H^2$.

%%%%%%%%%%%%%%%%

\begin{notn}
For a vector subspace $V\subset H^1$, 
let $A_V$ be the subalgebra of $\Hb$ 
generated by $V$. 
We write $B_V=\im(V\otimes V\to H^2)$.  
Then
$A_V= \F_p \oplus V\oplus B_V$, and every product involving $B_V$ and a
positive-degree element is zero. 
We introduce the following vector spaces similar to the spaces in the Koszul complex of a quadratic algebra. 
We write $\mu_V \colon V\otimes V \to H^2$, 
$x\otimes y \mapsto x \cup y$, 
and set $R_V \coloneqq \ker(\mu_V)$. 
For $m\ge1$, we define $\Krel_1(V) \coloneqq V$, $\Krel_2(V) \coloneqq R_V$, 
and, for $m\ge3$,
\begin{equation}\label{eq:Km2}
  \Krel_m(V) \coloneqq
  \bigcap_{j=0}^{m-2}
  V^{\otimes j}\otimes R_V\otimes V^{\otimes(m-2-j)}
  \subset V^{\otimes m}.
\end{equation}
In particular, 
$\Krel_3(V) = (R_V\otimes V)\cap(V\otimes R_V)$. 
\end{notn}

\begin{lemma}
\label{lem:deconcatenation}
For $i,j \ge 1$, 
there is a natural inclusion
\[
\Delta_{i,j} \colon  \Krel_{i+j}(V)\hookrightarrow \Krel_i(V)\otimes\Krel_j(V).
\]
These inclusions are compatible with iterated deconcatenation of tensor powers.
\end{lemma}

\begin{proof}
An element of $\Krel_{i+j}(V)$ lies in every subspace obtained by imposing the quadratic relation $R_V$ on two adjacent tensor factors.  
In particular, imposing the relations internal to the first $i$ factors places it in $\Krel_i(V)\otimes V^{\otimes j}$, 
and imposing those internal to the last $j$ factors places it in $V^{\otimes i}\otimes\Krel_j(V)$.  
Over a field, we have 
$(A\otimes Y)\cap(X\otimes B)=A\otimes B$ 
inside $X\otimes Y$ whenever $A\subset X$ and $B\subset Y$.  
Hence we get an inclusion into 
$\Krel_i(V)\otimes\Krel_j(V)$.  
Compatibility is inherited from the ordinary deconcatenation of tensor powers.
\end{proof}

%%%%%%%%%%%%

\begin{lemma}\label{lemma:K_VW}
If $W\subset V\subset H^1$, then  
$\Krel_m(W)=\Krel_m(V)\cap W^{\otimes m}$ for $m \ge 1$. 
\end{lemma}
\begin{proof}
It is clear that  $R_W=R_V\cap W^{\otimes2}$. 
More generally,  let $0 \le  j \le m-2$, 
and consider the quotient map 
$q_j\colon V^{\otimes m}\to V^{\otimes j}\otimes (V^{\otimes 2}/R_V)\otimes V^{\otimes(m-j-2)}$.   
Its kernel is $V^{\otimes j}\otimes R_V\otimes V^{\otimes(m-j-2)}$. 
The kernel of its restriction to $W^{\otimes m}$ 
is 
\[ 
W^{\otimes j}\otimes (R_V\cap W^{\otimes2})\otimes W^{\otimes(m-j-2)} = W^{\otimes j}\otimes R_W\otimes W^{\otimes(m-j-2)}. 
\] 
Intersecting these identities over $j$ 
gives $\Krel_m(W) = \Krel_m(V)\cap W^{\otimes m}$.
\end{proof}

%%%%%%%%%%%%

For $n \ge 3$, we define the linear map 
\begin{equation}\label{eq:DnV}
D_{n,V} \colon \Hom_{\Fp}(\Krel_{n-1}(V),H^1)  \longrightarrow  \Hom_{\Fp}(\Krel_n(V),H^2)
\end{equation}
by setting 
\begin{align*}
D_{n,V} (f)(x_1, \ldots, x_{n}) \coloneqq (-1)^{n} x_1 \cup f(x_2, \ldots, x_{n})  
+  (-1)^{n}f(x_1, \ldots, x_{n-1})\cup x_{n}. 
\end{align*}
%
%%%%%%

Now we show that this construction receives a map from Hochschild cohomology. 
A normalised Hochschild $n$-cochain on $A_V$ of internal degree $2-n$ has only one possible nonzero component 
consisting of maps of the form 
$V^{\otimes n} \to H^2$. 
There is no outgoing Hochschild differential, 
since the target of the next differential would lie in cohomological degree at least $3$.
A normalised Hochschild $(n-1)$-cochain of internal degree $2-n$ has two kinds of components:
\begin{enumerate}[label=\textnormal{(\roman*)}]
\item a component
$V^{\otimes(n-1)} \to H^1$,  
\item for each $0\le j\le n-2$, a component factoring through  
\[
V^{\otimes j}\otimes B_V\otimes V^{\otimes(n-2-j)} \to H^2.
\]
\end{enumerate}
The middle terms in the Hochschild differential in \eqref{eq:HH_differential} 
come from the components in (ii).  
Their joint annihilator is exactly $\Krel_n(V)$.
After restriction to $\Krel_n(V)$, the two endpoint terms depend 
only on the restriction of the component in (i) to $\Krel_{n-1}(V)$.  
This shows $\im(\partial^{n-1})|_{\Krel_{n}(V)} \subset \im(D_{n,V})$. 
Hence, for every $n \ge 3$, restriction to $\Krel_n(V)$ induces  
a well-defined homomorphism 
\begin{equation}\label{eq:HH-coker}
\iota_{n,V} \colon \HH^{n,2-n}(A_V,\Hb) \longrightarrow \coker(D_{n,V}), [f] \mapsto [f|_{\Krel_n(V)}]. 
\end{equation} 
Since every linear map $\Krel_n(V) \to H^2$ can be linearly extended to a map $V^{\otimes n} \to H^2$, 
$\iota_{n,V}$ is surjective. 
We show that this is actually an isomorphism. 

\begin{theorem}
\label{thm:HHcomparison} 
%We assume that $\Hb$ satisfies \eqref{eq:cd2-cohom}. 
Let $\Hb$ be a graded $\F_p$-algebra such that $H^0 = \F_p$ and $H^i=0$ for $i \ge 3$. 
For every vector subspace $V \subset H^1$ and 
for every $n \ge 3$, 
the map $\iota_{n,V}$ is an isomorphism 
which is functorial with respect to inclusions of vector subspaces in $H^1$. 
\end{theorem}

\begin{proof}
For $n \ge 3$, every $n$-cochain of internal degree $2-n$ is a map $V^{\otimes n}\to H^2$ and is automatically a cocycle.  
It remains to show that $\iota_{n,V}$ is injective. 
Let $[f]\in \HH^{n,2-n}(A_V,\Hb)$ belong to the kernel of $\iota_{n,V}$. 
Then $f|_{K_n(V)} = D_{n,V}(\varphi)$ 
for some linear map $\varphi \colon K_{n-1}(V) \to H^1$. 
We choose a linear extension 
$\gamma \colon V^{\otimes(n-1)}\to H^1$. 
After restriction to $K_n(V)$, 
all middle terms in $\partial^{n-1}\gamma$ vanish, 
while its endpoint terms are $D_{n,V}(\varphi)$. 
Hence $(f-\partial^{n-1}\gamma)|_{K_n(V)} =0$.  
It therefore suffices to show that every map $V^{\otimes n}\to H^2$ vanishing on $K_n(V)$ is a Hochschild coboundary.  
So, let $f \colon V^{\otimes n} \to H^2$ be a map such that $f|_{\Krel_n(V)} = 0$. 
For a subspace $S\subseteq V^{\otimes n}$, 
we write 
$S^\ann \coloneqq \{ f \in \Hom_{\F_p}(V^{\otimes n},H^2) \colon f|_S=0\}$.   
Let 
$X_j = V^{\otimes j}\otimes R_V\otimes V^{\otimes(n-j-2)}$ 
for $0\le  j \le n-2$.  
Then $K_n(V)=\bigcap_j X_j$. 
Since the intersection is finite, we have 
$K_n(V)^\ann = \sum_{j=0}^{n-2}X_j^\ann$.  
A map in $X_j^\ann$ factors through 
$V^{\otimes n}/X_j \cong V^{\otimes j}\otimes B_V\otimes V^{\otimes(n-j-2)}$.  
It is therefore the corresponding middle component of the Hochschild differential of a normalised $(n-1)$-cochain of internal degree $2-n$. 
For an $(n-1)$-cochain supported on a component with one $B_V$-factor, 
its endpoint terms vanish for degree reasons as $A_V^{>0} \cdot H^2 = 0 = H^2 \cdot A_V^{>0}$. 
Thus each map in $X_j^\ann$ is the middle component of a Hochschild boundary. 
This shows that $f$ is a Hochschild coboundary. 
This proves the injectivity of $\iota_{n,V}$. 
%
%%%%%%%%%%
%
Functoriality in $V$ follows from 
the compatibility of $D_{n,-}$ with the restriction along inclusions of vector spaces 
and the fact that the Hochschild restriction map corresponds to the induced map on cokernels. 
%
%%%%%%%%%%%%%%
\end{proof}

%%%%%%%%%%%%%
%
%%%%%%%%%%%%%

Now we consider the case $n=3$. 
The Hochschild cohomology class $[m_3]\in \HH^{3,-1}(\Hb,\Hb)$ 
is independent of the choices, 
and we call  
$[m_3]$ the \emph{canonical $A_3$-class of $\Cb$} following \cite[Section 5]{BKS}. 
The differential graded algebra $\Cb$ is $A_3$-formal if and only if the 
canonical $A_3$-class of $\Cb$ vanishes (see for example \cite[Theorem 3.9]{PQodd}). 
A map representing $[m_3]$ %The canonical $A_3$-class of $\Cb$ 
can be constructed as follows (see for example \cite[Section 5]{BKS}, \cite{Merkulov}). 
We choose an $\Fp$-linear map $f_1 \colon \Hb \to \Zb$ of degree $0$ 
which induces the identity after taking cohomology. 
We choose an $\Fp$-linear map $f_2 \colon \Hb \otimes \Hb \to \Cb$ of degree $-1$
such that 
$\delta f_2(x,y) = - f_1(x) \cup f_1(y) + f_1(x \cup y)$.  
We then define the $\F_p$-linear map $\kappa \colon \Hb \otimes \Hb \otimes \Hb \to \Cb$ of degree $-1$ 
by
\begin{align*}%\label{eq:def_kappa_3}
\kappa(x,y,z) \coloneqq (-1)^{|x|}f_1(x) \cup f_2(y,z) - f_2(x \cup y, z) + f_2(x, y \cup z) - f_2(x,y) \cup f_1(z).
\end{align*}
A direct computation shows that $\kappa(x,y,z)$ is a cocycle in $\Cb$. 
%\Gereon{think about signs!}  
% 
We let $[\kappa] \colon \Hb \otimes \Hb \otimes \Hb \to \Hb$ denote the $\F_p$-linear map of degree $-1$ defined by sending $x \otimes y \otimes z$ to 
the class $[\kappa(x,y,z)]$ in $\Hb$. 
The map $[\kappa]$ defines a Hochschild $(3,-1)$-cocycle.  
The class of $[\kappa]$ in $\HH^{3,-1}(\Hb,\Hb)$ is the canonical $A_3$-class $[m_3]$ of $\Cb$.  
Now let $V \subset H^1$ be a vector subspace. 
We denote by $[m_{3,V}] \in \HH^{3,-1}(A_V,\Hb)$ the image of $[m_3]$ 
under the map $\HH^{3,-1}(\Hb,\Hb) \to \HH^{3,-1}(A_V,\Hb)$ 
induced by the inclusion $A_V \into \Hb$, 
and we call $[m_{3,V}]$ the {\em restricted canonical $A_3$-class}. 
This terminology is justified by the following observation. 
\begin{lemma}\label{lemma:A_3_restricted_class}
The differential graded algebra $\Cb$ is weakly $A_3$-formal 
if and only if 
$[m_{3,V}]$ vanishes for every finite-dimensional vector subspace $V \subset H^1$. 
\end{lemma}
\begin{proof}
If $\Cb$ is weakly $A_3$-formal in the sense of Definition \ref{def:weak-An}, 
then $m_3|_{V^{\otimes 3}} = 0$, and hence its Hochschild class vanishes as well. 
Now we assume $[m_{3,V}]=0$. 
Then we can choose a normalised Hochschild $2$-cochain $h_V$ on $A_V$ such that
$\partial h_V = m_3|_{A_V^{\otimes3}}$.  
We then extend $h_V$ linearly to a normalised $2$-cochain on $\Hb$. 
Then we use $h_V$ to adjust the initial choice of the homotopy transfer  and hence the minimal structure such that 
the new $m_3$ restricts to zero on $V^{\otimes 3}$, see \cite[Proof of Theorem 2.7]{PQodd}. 
\end{proof}

Next, let $G$ be a profinite group, and let $\Cb=\Cb(G,\F_p)$ 
with cocycles $\Zb = \Zb(G,\Fp)$ and cohomology algebra $\Hb$.  
We assume $H^i = 0$ for $i \ge 3$. 
We identify $\Zh^1 = H^1 = \Hom_{\cont}(G,\Fp)$ 
with the vector space of continuous group homomorphism $G \to \Fp$. 
Let $V \subset H^1$ be a vector subspace. 
In this case, we can simplify the construction of the obstruction class for weak $A_3$-formality as follows. 
Let $\beta_2 \colon R_V \to \Ch^1$ be a linear map such that 
%\begin{align*} 
$\delta \beta_2(x,y) = - x \cup y$.  
%\end{align*}
%
We then define the $\F_p$-linear map $\Omega(\beta_2) \colon \Krel_3(V) \to \Zh^2$ by
\begin{align}\label{eq:def_D_V_general}
\Omega(\beta_2)(u,v,w) \coloneqq -u \cup \beta_2(v,w) - \beta_2(u,v) \cup w.
\end{align}
Again, a direct computation shows that $\Omega(\beta_2)(u,v,w)$ is a cocycle in $\Ch^2$. 
We write 
%\begin{align*}
$[\Omega(\beta_2)] \colon \Krel_3(V) \to H^2$  
%\end{align*}
for the resulting linear map to cohomology. 
If $\gamma \colon R_V \to \Zh^1$ is a cocycle-valued linear map, 
then we have 
$[\Omega(\beta_2 + \gamma)]-[\Omega(\beta_2)] = D_{3,V}[\gamma]$.  
We let $\ob_{3,V}$ 
denote the image of $[\Omega(\beta_2)]$ in $\coker(D_{3,V})$.   

\begin{lemma}
\label{lemma:two-step-comparison}
The isomorphism $\iota_{3,V}$ sends $[m_{3,V}]$ to $\ob_{3,V}$. 
\end{lemma}

\begin{proof}
We may assume that $[m_{3,V}]$ is constructed via maps $(f_1,f_2)$ as above 
with $f_1|_{H^1} \colon H^1 \xto{=} \Zh^1$ being the canonical identification. 
We set $\beta_2 \coloneqq f_2|_{R_V}$. 
By construction, $\iota_{3,V}$ is then given by taking the class of the restriction of $\kappa$ to $K_3(V)$. 
The two middle terms of $\kappa$ vanish on $K_3(V)$, and the remaining terms are exactly $\Omega(\beta_2)$. 
Hence the class of $[\Omega(\beta_2)]$ in $\coker(D_{3,V})$ 
equals both $\iota_{3,V}([m_{3,V}])$ and $\ob_{3,V}$. 
\end{proof}

%%%%%%%%%%%%%%%%

\begin{proposition}
\label{prop:global-two-step}
Let $G$ be a profinite group such that $H^i(G,\F_p) = 0$ for $i \ge 3$. 
Then $\Cb(G,\Fp)$ is weakly $A_3$-formal if and only if, 
for every finite-dimensional vector subspace $V \subset H^1$, 
we have $\ob_{3,V} = 0$. 
Moreover, 
$\Cb(G,\Fp)$ is $A_3$-formal if and only if $\ob_{3,H^1} = 0$. 
\end{proposition}
\begin{proof}
By Lemma \ref{lemma:two-step-comparison}, 
we have $\ob_{3,V} = 0$ if and only if $[m_{3,V}]=0$. 
Since $\Cb(G,\Fp)$ is weakly $A_3$-formal if and only if  $[m_{3,V}]=0$ 
for every finite-dimensional $V \subset H^1$, 
this proves the first assertion. 
For the second assertion, suppose that $\ob_{3,H^1} = 0$. 
By Lemma  \ref{lemma:two-step-comparison}, the restricted canonical class $[m_{3,H^1}]$ vanishes. 
Repeating the proof of Lemma \ref{lemma:A_3_restricted_class} with  $V=H^1$, 
we may change the transferred minimal structure so that $m_3|_{(H^1)^{\otimes 3}} = 0$. 
Since $H^0(G,F_p) =\F_p$ and $H^i(G,\F_p) = 0$ for $i \ge 3$, this is the only potentially nonzero component of $m_3$. 
Thus, $m_3=0$. 
The other implication is immediate.  
\end{proof}

%%%%%%%%%%%%%%%%
%
%%%%%%%%%%%%%%%%

\section{Embedding problems and weak \texorpdfstring{$A_n$}{An}-formality}\label{sec:finite-embedding-comparison}

In this section, we characterise  weak and genuine $A_n$-formality via embedding problems.
We let $G$ denote a profinite group. 

\begin{definition}
An {\em embedding problem} for $G$ is a solid diagram of continuous group homomorphisms 
\begin{align*}%\label{eq:def_embedding_problem}
\xymatrix{
 & G \ar[d]%^{\phi}
 \ar@{.>}[ld] %_{\widetilde{\phi}} %\ar@{.>}[ld]_{\widetilde{\phi}}
\\
B \ar@{->>}[r]^{\phi} & A}
\end{align*} 
in which $A$ and $B$ are topological groups, and  $\phi \colon B \twoheadrightarrow A$ is surjective.  
We refer to a continuous homomorphism $G \to B$ which makes the diagram commutative 
as a {\em solution} of the embedding problem. 
The {\em kernel} of the embedding problem is defined to be the kernel of $\phi$. 
We call the embedding problem {\it finite} if $A$ and $B$ are finite discrete groups. 
\end{definition} 

%%%%%%%%%%

\begin{remark} 
The literature on embedding problems often refers to a homomorphism $G \to B$ which makes the diagram commutative as 
a {\em weak} solution and reserves the term {\em solution}  for a weak solution which is  {\em surjective}. 
Since surjectivity of the solution does not matter in this paper, 
we use the shorter terminology and do not make the above distinction. 
\end{remark}

%%%%%%%%%%%%%%%%

\begin{notn}\label{notn:dual}
We let $p$ be a fixed prime number, and $G$ be a profinite group. 
We recall from \cite[Definition 1.1.10]{NSW} that the Pontryagin dual 
of a Hausdorff, abelian and locally compact group $A$ is the group $A^{\vee} = \Hom_{\mathrm{cont}}(A,\R/\Z)$ 
of continuous group homomorphisms from $A$ to $\R/\Z$ where the latter is equipped with the quotient topology inherited from $\R$. 
The group $A^{\vee}$ is considered as a topological group with the compact-open topology. 
We recall that Pontryagin duality states that the canonical homomorphism $A \to (A^{\vee})^{\vee}$ is an isomorphism of topological groups 
(see \cite[Theorem 1.1.11]{NSW} and the references therein). 
If $V$ is an $\F_p$-vector space, 
then $V$, considered with the discrete topology, is a Hausdorff, abelian and locally compact group. 
Every group homomorphism $V \to \R/\Z$ has image in the $p$-torsion subgroup $(\R/\Z)[p] =\frac{1}{p}\Z/\Z$. 
Using the identification of additive groups 
$\frac{1}{p}\Z/\Z \xto{\cong} \F_p$, 
$\frac{a}{p} + \Z \mapsto \overline{a}$, 
we obtain a natural identification $V^\vee \cong \Hom_{\F_p}(V,\F_p)$. 
We equip $\Hom_{\F_p}(V,\F_p)$ with the compact-open topology. 
Since $V$ is discrete, its compact subsets are finite. 
Hence the compact-open topology on $V^\vee$ is the topology of pointwise convergence on finite subsets of $V$.
In particular, $V^\vee$ is a compact Hausdorff totally-disconnected topological $\F_p$-vector space. 
Hence, for a discrete $\F_p$-vector space, the Pontryagin dual equals the algebraic dual equipped with the topology of pointwise convergence.
If $V$ is finite-dimensional, then $V^\vee$ is finite and its topology is discrete. 
\end{notn}

\begin{notn}
For a topological $\F_p$-vector space $M$, 
we write $\Cb(G,M)$ for the complex of continuous inhomogeneous cochains, 
where we regard $M$, unless otherwise stated, as a trivial $G$-module.
We denote the cocycles by $\Zb(G,M)$ and the corresponding cohomology groups by $\Hb(G,M)$.  
For $M = \F_p$, we often just write $\Cb$, $\Zb$, and $\Hb$. 
In this case, $\Hb$ is the usual graded cohomology algebra.
\end{notn}

%%%%%%%%%%%%%%%%%%%

Throughout this section, 
we let $G$ be a profinite group such that $H^i(G,\F_p) = 0$ for $i \ge 3$, 
and we let $V \subseteq H^1(G,\F_p)$ be a vector subspace. 

\begin{definition}
We equip $V$ and every subspace $K_m(V)\subseteq V^{\otimes m}$, with the discrete topology. 
We define 
\[
P_m(V) \coloneqq \Krel_m(V)^\vee  % = \Hom_{\Fp}(\Krel_m(V),\Fp) 
\]
and equip it with the compact-open topology. 
Thus, 
$P_m(V)$ is a compact Hausdorff totally disconnected topological
$\Fp$-vector space. 
In particular, if $V$ is finite-dimensional, then $P_m(V)$ is finite and its topology is discrete.  
Dualising the deconcatenation inclusion
$\Delta_{i,j} \colon K_{i+j}(V) \to K_i(V)\otimes K_j(V)$ 
of Lemma \ref{lem:deconcatenation} 
defines a bilinear product  
\begin{align}\label{eq:product_on_P}
P_i(V)\times P_j(V) \to P_{i+j}(V),
~ 
(\lambda_i \lambda_j)(x) \coloneqq (\lambda_i \otimes \lambda_j)\left(\Delta_{i,j}(x)\right). 
\end{align}
\end{definition}

%%%%%%%%%%%%

\begin{lemma}
The product \eqref{eq:product_on_P} is continuous. 
\end{lemma}
\begin{proof}
For each fixed $x\in K_{i+j}(V)$, the 
tensor $\Delta_{i,j}(x)$ is a finite sum
$\Delta_{i,j}(x)=\sum_{r=1}^s a_r\otimes b_r$, 
so  
$(\lambda_i \lambda_j)(x)=\sum_{r=1}^s\lambda_i(a_r)\lambda_j(b_r)$  
depends on only finitely many coordinates of $\lambda_i$ and $\lambda_j$. 
Continuity now follows from the definition of the topology on both sides. 
\end{proof}

%%%%%%%%%%%%%%%%%%%

\begin{lemma}\label{lemma:CP_identifications}
Let $V \subset H^1(G,\Fp)$ be a vector subspace. 
For every $q \ge 0$ and $m\ge 1$, we have natural identifications  
\begin{align}\label{eq:identification_CPnV}
\Ch^q(G,P_m(V)) \cong \Hom_{\Fp} (K_m(V), \Ch^q(G,\Fp)), 
\end{align}
and 
$H^q(G,P_m(V)) \cong \Hom_{\Fp}(K_m(V), H^q(G,\Fp))$. 
\end{lemma}
\begin{proof}
Since $P_m(V)$ is equipped with the topology of pointwise convergence, 
a map 
$c \colon G^q \to P_m(V)$ 
is continuous if and only if every coordinate
$g \mapsto c(g)(x)$ for $x \in K_m(V)$ 
is continuous. 
This proves \eqref{eq:identification_CPnV}. 
Since the functor $\Hom_{\Fp}(K_m(V),-)$ is exact over a field, 
taking cohomology implies the second isomorphism. 
\end{proof}

%%%%%%%%%%%%%%%%

%%%%%%%%%%%%%%%%

Let $\Fh(V)$ denote the directed poset of finite-dimensional 
$\F_p$-subspaces of $V$.  
Every element of $(V)^{\otimes m}$ is a finite sum of pure tensors and therefore belongs to $W^{\otimes m}$ 
for some $W \in  \Fh(V)$.  
Hence Lemma \ref{lemma:K_VW}  %\eqref{eq:intersection-property} 
implies
\begin{equation*}%\label{eq:filtered-union}
K_m(V) = \bigcup_{W\in\mathcal F(V)}K_m(W).
\end{equation*}
\begin{lemma}
For every $m \ge 1$, 
inclusions of vector spaces induce an isomorphism of topological vector spaces 
\begin{equation}\label{eq:P-limit}
P_m(V)\xrightarrow{\cong}
\varprojlim_{W\in\mathcal F(V)}P_m(W). 
\end{equation}
\end{lemma}
\begin{proof}
Since $P_m(W) = K_m(W)^\vee$,  
a compatible family of functionals $(\lambda_W)_W$ with $\lambda_W \in K_m(W)^\vee$ 
defines a unique functional on the union $K_m(V)$. 
Conversely, a functional on  $K_m(V)$ restricts to such a compatible family. 
Thus, we get an isomorphism of vector spaces  
$P_m(V)\xrightarrow{\cong}  \varprojlim_{W\in\mathcal F(V)}P_m(W)$.  
Moreover, this is a homeomorphism since 
the compact-open topology on the dual of a discrete vector space 
is the topology of pointwise convergence on finite subsets, 
which is the 
inverse-limit topology on the right-hand side of \eqref{eq:P-limit}.
\end{proof}

%%%%%%%%%%%%%%%%%

\begin{definition}
For $n \ge 2$, we define 
the graded $\Fp$-vector space
$\uA_n(V) \coloneqq \bigoplus_{m=1}^n P_m(V)$,  
and consider $\uA_n(V)$ with the products \eqref{eq:product_on_P} 
and with every product of total degree $>n$ declared zero. 
We equip $\uA_n(V)$ with the finite product topology. 
We define % the group 
\[
\UA_n(V) \coloneqq 1+\uA_n(V)
\] 
with multiplication $(1+x)(1+y)=1+x+y+xy$ 
and equip $\UA_n(V)$ with the topology defined by the bijection 
$\uA_n(V) \to \UA_n(V)$, $x\mapsto1+x$. 
\end{definition}

%%%%%%%%%%%%

\begin{lemma}
\label{lemma:nilpotent-algebra}
For $n \ge 2$, $\uA_n(V)$ is an associative topological nilpotent $\Fp$-algebra, 
and $\UA_n(V)$ is a compact Hausdorff totally disconnected topological group.  
If $V$ is finite-dimensional, then $\UA_n(V)$ is a finite discrete $p$-group. 
\end{lemma}

\begin{proof}
Associativity follows from dualising Lemma \ref{lem:deconcatenation}. 
Nilpotence follows from the grading.  
Since the products on the $P_m(V)$ are continuous, 
the operations on $\uA_n$ and $\UA_n(V)$ are continuous. 
We note that taking inverses in $\UA_n(V)$ is continuous 
since $\uA_n(V)^{n+1}=0$ and
$(1+x)^{-1}=1-x+x^2-\cdots+(-1)^n x^n$.  
Finally, if $V$ is finite-dimensional, then 
$1+\uA_n(V)$ has cardinality $p^{\dim\uA_n(V)}$, so it is a finite $p$-group.
\end{proof}

\begin{remark}\label{rem:kernel_is_elementary_abelian}
In degree one, we have 
$P_1(V)=V^\vee$. 
We write $N_{n,V} \coloneqq 1+\bigoplus_{m=2}^nP_m(V)$.  
%which is an additive group. 
%
It is a closed normal subgroup of $\UA_n(V)$, 
equal to the kernel of $\UA_n(V) \twoheadrightarrow V^\vee$. 
We then obtain a short exact sequence of topological groups 
\begin{equation*}%\label{eq:higher-extension}
  1\longrightarrow N_{n,V}\longrightarrow\UA_n(V)
  \longrightarrow V^\vee \longrightarrow1.
\end{equation*}
For $n=3$, multiplication on $N_{3,V}$ agrees with addition on $P_2(V) \oplus P_3(V)$, %the additive kernel 
so $N_{3,V}$ is elementary abelian.  
For $n\ge4$, however, $N_{n,V}$ need not be abelian. 
\end{remark}

%%%%%%%%%%%%%%%%

\begin{lemma}
For every $n \ge 2$, 
inclusions of vector spaces induce an isomorphism of topological groups
\begin{equation}\label{eq:U-limit}
\UA_n(V) \xrightarrow{\cong} \varprojlim_{W\in\mathcal F(V)} \UA_n(W).
\end{equation}
In particular, $\UA_n(V)$ is a pro-$p$ group.  
\end{lemma}
\begin{proof} 
Since the deconcatenation maps are natural with respect to inclusions of vector
spaces, isomorphism \eqref{eq:P-limit} respects the
products on the graded pieces. 
%
%Since only the finitely many degrees $1,\ldots,n$ occur, 
Since only finitely many degrees occur, 
they combine to an isomorphism of topological nilpotent
algebras
$\uA_n(V)\xrightarrow{\cong}
\varprojlim_{W\in\mathcal F(V)} \uA_n(W)$, 
and hence to an isomorphism of topological groups
$\UA_n(V)\xrightarrow{\cong} \varprojlim_{W\in\mathcal F(V)} \UA_n(W)$.
\end{proof}

%%%%%%%%%%%%%

Now we can relate weak $A_n$-formality to the existence of a solution of certain embedding problems. 
We let $V \subset H^1(G,\Fp)$ be a vector subspace. 
We first consider the cases $n=2$ and $n=3$ and will then explain how to generalise to arbitrary $n$. 

\begin{definition}
The {\em tautological character} is defined as the group homomorphism 
%\begin{align*}
$\chi_V \colon G \to V^\vee$, 
$\chi_V(g)(v) \coloneqq  v(g)$. 
%\end{align*}
%
\end{definition}

%%%%%%%%%%%

\begin{lemma}\label{lemma:chi_V_continuous_surjective}
The group homomorphism $\chi_V$ is continuous  and surjective. 
\end{lemma}
\begin{proof}
For each $W \in \Fh(V)$, the composition 
\[
G \xrightarrow{\chi_V} V^\vee \xrightarrow{\cong} \varprojlim_{W\in\mathcal F(V)} W^\vee \to W^\vee  
\]
of $\chi_V$ with \eqref{eq:P-limit} and the canonical projection 
equals $\chi_W$. 
Since \eqref{eq:P-limit} is a homeomorphism and each $\chi_W$ is continuous, 
we get that $\chi_V$ is continuous. 
Now we show that $\chi_V$ is surjective. 
Let $I \coloneqq \im \chi_V$, and let 
$I^\ann = \{ v \in V : \lambda(v) = 0 ~\text{for every}~ \lambda \in I\}$ denote its annihilator. 
By definition of $\chi_V$, 
$v \in I^\ann$ if and only if $v(g)=0$ for every $g \in G$. 
However, since $V \subseteq H^1(G,\F_p)$, 
$v$ itself is a continuous group homomorphism $G \to \F_p$. 
If it vanishes for every $g \in G$, then $v = 0$. 
Thus, we get $I^\ann = 0$. 
By Pontryagin duality, 
the closure of $I$ is the annihilator of $I^\ann$. 
Since $I^\ann = 0$, 
we get that the closure of $I$ is $V^\vee$.   
Since $G$ is a profinite group, it is compact. 
Since $\chi_V$ is continuous, $I = \chi_V(G)$ is compact as well. 
Since $V^\vee$ is Hausdorff, the compact subset $I \subseteq V^\vee$ is closed in $V^\vee$. 
This shows that $I$ is a dense and closed subset of $V^\vee$, 
and hence $\im \chi_V = V^\vee$. 
\end{proof}

\begin{notn}\label{notn:chiV_cup_chiV}
We write $\chi_V \cup \chi_V \colon G\times G\to (V\otimes V)^\vee$ 
for the map defined  by $\bigl((\chi_V \cup \chi_V)(g_1,g_2)\bigr)(v\otimes w) = v(g_1)w(g_2)$. 
Composing with the restriction map 
$(V\otimes V)^\vee \to R_V^\vee=P_2(V)$ 
gives a continuous $P_2(V)$-valued $2$-cocycle, 
which we again denote by $\chi_V \cup \chi_V$, 
by slight abuse of notation. 
\end{notn}

\begin{lemma}\label{lemma:quadratic_lift_exists}
The embedding problem $G \xrightarrow{\chi_V}  V^\vee \twoheadleftarrow \UA_2(V)$ 
has a solution. 
Moreover, every solution $\rr_2$ is of the form $\rr_2 = 1 + \chi_V + b_2$
%is in bijection with the set of 
for some cochain $b_2 \in \Ch^1(G, P_2(V))$ 
such that $\delta b_2 =- \chi_V \cup \chi_V$. 
\end{lemma}
\begin{proof}
It follows from the definition of the group structure on $\UA_2(V)$ 
and a direct computation, 
using that the products $P_1(V) P_2(V)$,  $P_2(V) P_1(V) $, and $P_2(V) P_2(V)$ are trivial for $\UA_2(V)$, 
that every solution $\rr_2$ is of the form $\rr_2 = 1 + \chi_V + b_2$
%is in bijection with the set of 
for some cochain $b_2 \in \Ch^1(G, P_2(V))$ 
such that $\delta b_2 = - \chi_V \cup \chi_V$. 
To finish the proof, % of the first assertion of the lemma, 
it suffices to observe that such a cochain $b_2$ always exists 
since $\mu_V$ is trivial on $R_V$ by definition of the latter. 
\end{proof}

%%%%%%%%%%%%

%%%%%%%%%%

Now we consider the embedding problem 
\begin{align*}%\label{eq:embedding_problem_UUU}
\mathbf E^{(3)}_V \colon   G \xrightarrow{\chi_V}  V^\vee \twoheadleftarrow \UA_3(V).
\end{align*}
Let $b_2 \in \Ch^1(G,P_2(V))$ such that $\delta b_2 = -\chi_V \cup \chi_V$. 
%and let $\brr$ be the corresponding lift $G \to \bUU(V)$. 
%
We then define the $P_3(V)$-valued cocycle 
\begin{align}\label{eq:def_omega_b}
\omega(b_2) \coloneqq - \chi_V\cup b_2 - b_2 \cup \chi_V\in \Zh^2(G,P_3(V)). 
\end{align}
We note that $\omega(b_2)$ is indeed a cocycle 
since $\delta(\chi_V \cup b_2) = \chi_V \cup \chi_V \cup \chi_V$ 
and $\delta(b_2 \cup \chi_V) = - \chi_V \cup \chi_V \cup \chi_V$. 

\begin{lemma}\label{lemma:cubic_lift_cochain}
Let $\rr_2 = 1 + \chi_V + b_2 \colon G \to \UA_2(V)$ be a solution 
of the embedding problem $G \xrightarrow{\chi_V}  V^\vee \twoheadleftarrow \UA_2(V)$. 
%$b_V \in \Ch^1(G, P_V)$  such that $\delta b_V =- \chi_V \cup \chi_V$. 
%
The embedding problem $\mathbf E^{(3)}_V \colon G \xrightarrow{\chi_V}  V^\vee \twoheadleftarrow \UA_3(V)$ 
has a solution which extends $\rr_2$ 
if and only if there is a cochain $b_3 \in \Ch^1(G, P_3(V))$ such that $\delta b_3 = \omega(b_2)$. 
\end{lemma}
\begin{proof}
This follows again from a direct computation, using that the higher products vanish in $\uA_3(V)$, 
and the definition of the group structure on $\UA_3(V)$. 
\end{proof}

%%%%%%%%%%%%%%

Let $b_2 \in \Ch^1(G,P_2(V))$ such that $\delta b_2 = -\chi_V \cup \chi_V$, 
and let $\omega(b_2)$ be as in \eqref{eq:def_omega_b}. 
Let $[\omega(b_2)]$ denote the corresponding cohomology class in $H^2(G,P_3(V))$. % \cong \Hom_{\F_p}(K_V,H^2)$. 
%where we recall the notation $H^2 = H^2(G,\F_p)$. 
%
By Lemma \ref{lemma:cubic_lift_cochain}, 
the class $[\omega(b_2)]$ in $H^2(G,P_3(V)) \cong \Hom_{\F_p}(\Krel_3(V),H^2)$  
is an obstruction to the existence of a lift to a solution of $\mathbf E^{(3)}_V$. %\eqref{eq:embedding_problem_UUU}. 
Now let $z \in \Zh^1(G,P_2(V))$ be a $P_2(V)$-valued cocycle. 
Let $[z] \in H^1(G,P_2(V))$ %= \Hom_{\F_p}(R_V, H^1)$ 
denote the corresponding cohomology class. 
Then 
$[\omega(b_2+z)]-[\omega(b_2)] = D_{3,V}[z]$,  
where $D_{3,V}([z])$ is defined as in \eqref{eq:DnV} %\eqref{eq:def_D_V,z_general} 
using Lemma \ref{lemma:CP_identifications}. 
Let 
$\ob_{3,V} \in \coker (D_{3,V})$ % \colon \Hom_{\F_p}(R_V,H^1) \to \Hom_{\F_p}(\Krel_3(V),H^2))$  
denote the image of $[\omega(b_2)]$ in the cokernel of $D_{3,V}$.   
Since a cochain $b_2 \in \Ch^1(G,P_2(V))$ such that $\delta b_2 = -\chi_V \cup \chi_V$ 
always exists by Lemma \ref{lemma:quadratic_lift_exists}, 
we have proven the following result. 

\begin{lemma}
\label{lemma:finite-simultaneous-lifting}
The embedding problem $\mathbf E^{(3)}_V$ %in \eqref{eq:embedding_problem_UUU} 
has a solution if and only if $\ob_{3,V}=0$. 
\end{lemma}

%%%%%%%%%%%

\begin{proposition}
\label{prop:two-step-comparison_finite_embedding}
The embedding problem $\mathbf E^{(3)}_V$ % in \eqref{eq:embedding_problem_UUU} 
has a solution if and only if the restricted canonical $A_3$-class vanishes. 
\end{proposition}
\begin{proof}
By Proposition \ref{prop:global-two-step}, % \ref{prop:two-step-comparison}, 
the restricted canonical $A_3$-class vanishes if and only if $\ob_{3,V} = 0$. 
Lemma \ref{lemma:finite-simultaneous-lifting} shows that 
$\ob_{3,V} = 0$ if and only if $\mathbf E^{(3)}_V$ has a solution. 
\end{proof}

%%%%%%%%%%%%%%%%%

For $n \ge 3$, we now consider the embedding problem 
\begin{equation*}%\label{eq:EPn}
\mathbf E^{(n)}_V \colon 
G\xrightarrow{\chi_V} V^\vee \twoheadleftarrow \UA_n(V). 
\end{equation*}

\begin{lemma}\label{lemma:nic_lift_cochain}
A solution $\rho \colon G \to \UA_n(V)$ of $\mathbf E^{(n)}_V$ % \eqref{eq:EPn} 
is of the form %has a unique expression
\[
\rho = 1+b_1+b_2+\cdots+b_n,
\]
where $b_1=\chi_V$, and $b_r \in \Ch^1(G,P_r(V))$ for $2\le r\le n$ %$m \ge 2$ 
such that 
\begin{equation}\label{eq:MC}
\delta b_r = -
\sum_{i = 1}^{r-1} b_i \cup b_{r-i} ~ \text{for} ~ 2\le r \le n,
\end{equation}
where the cup-product is followed by the multiplication
$P_i(V)\otimes P_j(V)\to P_r(V)$.

\end{lemma} 
\begin{proof}
This follows again from a direct computation, using that the higher products vanish in $\uA_n(V)$, 
and the definition of the group structure on $\UA_n(V)$. 
\end{proof}

%%%%

We now compare the recursive equations \eqref{eq:MC} with the ordinary obstruction theory 
for the vanishing of the minimal $A_\infty$-operations. 

\begin{notn}
Let $r \ge 3$. 
Let $\beta = (\beta_1, \ldots, \beta_{r-1})$ be a family of linear maps with
$\beta_1\colon V \to \Zh^1(G,\F_p)$ the restriction of the canonical identification $H^1 \xto{=} \Zh^1$, 
and, for $2\leq i<r$,
let $\beta_i \colon K_i(V) \to \Ch^1(G,\F_p)$.  
We then define the linear map 
\begin{align*}%\label{eq:notation_S}
S_j (\beta) \coloneqq \sum_{i=1}^{j-1} (\beta_i \cup \beta_{j-i}) \circ \Delta_{i,j-i} 
\colon K_j(V) \to \Ch^2(G,\Fp) ~ \text{for} ~ 2 \le j \le r.
\end{align*}
Here $S_j(\beta)$ depends only on $\beta_1, \dots, \beta_{j-1}$, 
so the same notation is used when $\beta$ is a longer family. 
We say that the family $\beta$ is {\em $V_r$-good} if the maps $\beta_i$ satisfy 
$\delta\beta_i = -S_i(\beta)$ for $2 \le i <r$. 
In this case, a direct computation gives $\delta S_r(\beta) = 0$, 
and we then write 
$\Omega_r(\beta) \coloneqq \varepsilon_rS_r(\beta) \colon K_r(V) \to \Zh^2(G,\F_p)$, 
where we set  $\varepsilon_j \coloneqq (-1)^{\lfloor (j-1)/2 \rfloor}$ so that $\varepsilon_1 = \varepsilon_2 = 1$. 
\end{notn}

%%%%%%%%%%%%

\begin{remark}
The choice of the sign $\varepsilon_r$ stems from the sign conventions for $A_{\infty}$-morphisms in \cite[Section 3.4]{Keller}. 
The Koszul sign rule turns the sign for $f_i \cup f_j$ on degree one classes to be $(-1)^{ij+1}$, 
%The sign 
and $\varepsilon_r$ is chosen such that $\varepsilon_{i+j} = (-1)^{ij+1}\varepsilon_i \varepsilon_j$. 

\end{remark}

%%%%%%%%%%%%

\begin{lemma}
\label{lem:truncated-twisting-cochain}
Under the natural identification \eqref{eq:identification_CPnV}, 
there is a bijection between the set of cochains $b_r \in \Ch^1(G,P_r(V))$ for $1\le r \le n$ 
with $b_1 = \chi_V$ and satisfying  \eqref{eq:MC} 
and the set of $V_{n+1}$-good families $\beta = (\beta_1, \ldots, \beta_{n})$. 
\end{lemma}

\begin{proof}
This follows from dualising the maps involved. 
\end{proof}

%%%%%%%%%%%%%%%%%%

\begin{definition} 
Let $\beta$ be an $V_r$-good family. 
We say that $\Hb$ is {\em $V_r^{\beta}$-nice} if 
there is a minimal $A_{\infty}$-structure on $\Hb$ 
together with a strictly unital $A_\infty$-quasi-isomorphism 
$F=(f_1,f_2,\ldots) \colon \Hb \to \Cb$  
such that
%\begin{equation}\label{eq:lower-operations-vanish}
$m_3|_{A_V^{\otimes3}} = \cdots =m_{r-1}|_{A_V^{\otimes(r-1)}} = 0$,   
%\end{equation}
where $f_1$ induces the identity after taking cohomology and 
$f_1|_{H^1} \colon H^1 \xto{=} \Zh^1$ is the canonical identification, 
and $f_i|_{K_i(V)} = \varepsilon_i \beta_i$ for $1\leq i < r$. 
\end{definition}

%%%%%%%%%%%%%%%%%%%%

\begin{proposition}
\label{prop:defining-system-comparison}
Let $r\ge 3$, 
and let $\beta = (\beta_1, \dots, \beta_{r-1})$ be a $V_r$-good family. 
Then $\Hb$ is $V_r^{\beta}$-nice 
with $m_r|_{K_r(V)} = [\Omega_r(\beta)]$ in 
$\Hom_{\F_p}(K_r(V),H^2)$. 
With $\beta_1,\dots,\beta_{r-2}$ fixed, 
the image of $[\Omega_r(\beta)]$ in $\coker(D_{r,V})$ is
independent of the choice of $\beta_{r-1}$.   
\end{proposition}

\begin{proof}
We prove the first assertion by induction on $r$. 
The assertion for $r=3$ is the concrete
description of the restricted canonical $A_3$-class. 
That is, we choose an initial transferred minimal structure and its associated strictly unital $A_{\infty}$-quasi-isomorphism 
$(f_1,\widetilde{f}_2, \widetilde{f}_3, \ldots)$. 
Since both $\widetilde{f}_2|_{R_V}$ and $\beta_2$ have differential   $-\beta_1 \cup \beta_1$, 
their difference is cocycle-valued. 
Then we extend this difference to a normalised cocycle-valued cochain, 
and adjust the initial $\widetilde{f}_2$ to get a linear map $f_2 \colon (\Hb)^{\otimes 2} \to \Cb$ of degree $-1$ 
such that $f_2|_{R_V} = \beta_2$. 
Now we assume the assertion has been proved % through order $r-1$. 
for $r-1$. 
Thus, after replacing the transferred minimal structure by a strictly unital
$A_\infty$-isomorphism whose linear component is the identity, we may
assume that
$m_3|_{A_V^{\otimes3}} = \cdots =m_{r-2}|_{A_V^{\otimes(r-2)}} = 0$, 
and that the components of the associated $A_\infty$-quasi-isomorphism
satisfy
$f_i|_{K_i(V)} = \varepsilon_i  \beta_i$ for $1 \le i < r-1$.   
%
%%%%%
%
We first explain how to realise the prescribed component in arity $r-1$. 
By the induction hypothesis, 
we have $m_{r-1}|_{K_{r-1}(V)} = [\Omega_{r-1}(\beta)]$ in $\Hom_{\F_p}(K_{r-1}(V),H^2)$, 
and 
$\delta\beta_{r-1} = -S_{r-1}(\beta)$.
Hence we have 
$\Omega_{r-1}(\beta) = \varepsilon_{r-1}S_{r-1}(\beta)  = - \varepsilon_{r-1}\delta \beta_{r-1}$. 
This shows that $\Omega_{r-1}(\beta)$ is exact, 
and hence $m_{r-1}|_{K_{r-1}(V)} = 0$ as a map to $H^2$. 
We can now use the support argument from the proof of Theorem \ref{thm:HHcomparison}:   
a Hochschild cochain vanishing on $K_{r-1}(V)$ is a Hochschild coboundary of cochains 
supported on components with one $B_V$-factor. 
Such cochains vanish on $V^{\otimes(r-2)}$, so the already prescribed restrictions of lower morphism components are preserved.
Thus, we can assume $m_{r-1}|_{A_V^{\otimes(r-1)}} = 0$ as well. 
%
%%%%%

Now we use the $(r-1)$-st identity of an $A_\infty$-morphism. 
By the induction hypothesis for the lower components, both
$f_{r-1}|_{K_{r-1}(V)}$ 
and $\varepsilon_{r-1} \beta_{r-1}$  
have differential 
$- \varepsilon_{r-1}S_{r-1}(\beta)$.  
We then get 
$\delta(f_{r-1}|_{K_{r-1}(V)}) = \varepsilon_{r-1}\delta\beta_{r-1}$. 
Their difference
$z_{r-1} \coloneqq \varepsilon_{r-1}\beta_{r-1}-f_{r-1}|_{K_{r-1}(V)}$  
therefore takes values in $\Zh^1$.
We choose a normalised cochain $\widetilde{z}_{r-1}\colon (\Hb)^{\otimes(r-1)} \to \Cb$ 
of degree $2-r$ whose restriction to $K_{r-1}(V)$ is $z_{r-1}$, 
whose values lie in $\Zb$, 
and which vanishes whenever one input is the unit.
Replacing $f_{r-1}$ by
$f_{r-1}+\widetilde{z}_{r-1}$
does not change the already established $A_\infty$-identities through order
$r-1$, since $\delta\widetilde{z}_{r-1}=0$. 
However, 
it changes the $r$-th source
operation by the Hochschild coboundary determined by
$\widetilde{z}_{r-1}$. 
Hence the lower operations remain zero on $A_V$, while the new component satisfies
$f_{r-1}|_{K_{r-1}(V)} = \varepsilon_{r-1} \beta_{r-1}$.  
%
%%%%%%%%%%%

Now we compute the obstruction in arity $r$. 
Consider the $r$-th
$A_\infty$-morphism identity for
$F \colon \Hb \to \Cb$. 
After restriction to $K_r(V)$, terms containing an insertion of $m_2$
vanish because every adjacent pair in $K_r(V)$ belongs to
$R_V$. %=\ker(V\otimes V \to H^2)$.  
Terms containing one of $m_3,\ldots,m_{r-1}$ vanish 
since $m_3|_{A_V^{\otimes3}} = \cdots =m_{r-1}|_{A_V^{\otimes(r-1)}} = 0$. 
%by \eqref{eq:lower-operations-vanish}. 
%
Hence the only source operation which remains
is $m_r|_{K_r(V)}$.   
The terms on the target side which remain are the products
$(f_i \cup f_{r-i})\circ\Delta_{i,r-i}$ for $1 \le i \le r-1$.   
Using $f_i|_{K_i(V)} = \varepsilon_i \beta_i$ %\eqref{eq:realise-beta} 
and 
$(-1)^{ij+1} \varepsilon_i \varepsilon_j = \varepsilon_{i+j}$, 
the restricted identity therefore gives, 
\begin{equation*}%\label{eq:(E)} %\label{eq:restricted-morphism-identity}
 f_1\bigl(m_r|_{K_r(V)}\bigr)  = \delta\bigl(f_r|_{K_r(V)}\bigr) +  \Omega_r(\beta).  %\Omega_r(\beta_1,\ldots,\beta_{r-1}).
\end{equation*}
Applying $\delta$ to both sides and using $\delta^2=0$ shows that
$\delta\Omega_r(\beta)=0$. 
Thus, $\Omega_r$ takes values
in $\Zh^2$. %, proving \textup{(ii)}.
After passing to cohomology we get 
\begin{equation*}%\label{eq:restricted-morphism-identity}
m_r|_{K_r(V)} 
=  [\Omega_r(\beta)] ~ \text{in} ~ \Hom_{\F_p}(K_r(V),H^2).
\end{equation*}
%
%%%%%%%%%%%%

Now suppose that $\beta' = (\beta'_1, \dots, \beta'_{r-1})$ is another $V_r$-good family, 
and assume that
$\beta_i' = \beta_i$ for $1 \leq i \le r-2$.  %
We write $z \coloneqq \beta_{r-1}'-\beta_{r-1}$. 
The defining-system equation in arity $r-1$ has the same right-hand side for
$\beta_{r-1}'$ and $\beta_{r-1}$. 
Subtracting the two equations therefore gives
$\delta z=0$. 
Hence $z$ takes values in $\Zh^1$.
Replacing the $\beta_i$ by the $\beta'_i$ changes 
only the two summands of $\Omega_r$ involving $\beta_{r-1}$. 
Thus, we get 
%\begin{align*}
$\Omega'_r(\beta') -\Omega_r (\beta) = \varepsilon_r \left( (\beta_1\cup z)\circ\Delta_{1,r-1} + (z\cup \beta_1)\circ\Delta_{r-1,1} \right)$.   
%\end{align*}
%
Since 
$D_{r,V}[z] =(-1)^r \left( (\beta_1\cup  z)\circ \Delta_{1,r-1} + (z\cup \beta_1)\circ\Delta_{r-1,1} \right)$, 
after passing to cohomology, we get $[\Omega'_r(\beta') -\Omega_r (\beta)] = (-1)^{r} \varepsilon_r D_{r,V}([z])$. 
This finishes the proof. 
\end{proof}

%%%%%%%%%%%%%%%%%

Now we can prove a key technical result. 

\begin{theorem}
\label{thm:finite-lifting}
Let $G$ be a profinite group such that $H^i(G,\F_p) = 0$ for $i \ge 3$, 
and we let $V \subset H^1(G,\F_p)$ be a vector subspace. 
Let $n\ge3$. 
Then the following are equivalent:
\begin{enumerate}
\item[\textnormal{(a)}] the minimal $A_\infty$-structure on $\Hb$ can be chosen 
such that \\ 
$m_3|_{A_V^{\otimes 3}} = \dots = m_n|_{A_V^{\otimes n}} = 0$; 
\item[\textnormal{(b)}]  there exist cochains $b_1, b_2,\dots,b_n$ satisfying \eqref{eq:MC} with $b_1 = \chi_V$; 
\item[\textnormal{(c)}]  the embedding problem
\[
\mathbf E^{(n)}_V \colon G\xrightarrow{\chi_V} V^\vee \twoheadleftarrow\UA_n(V)
\]
has a solution.
\end{enumerate}
\end{theorem}

\begin{proof}
The equivalence of (b) and (c) is given by Lemma \ref{lemma:nic_lift_cochain}. 
It remains to prove the equivalence of
(a) and (b).  
Assume first that (a) holds. 
We choose a strictly unital minimal
$A_\infty$-structure
$(H^\bullet,m_2,m_3,\ldots)$ 
and a strictly unital $A_\infty$-quasi-isomorphism
$F=(f_1,f_2,f_3,\ldots) \colon \Hb \to \Cb$ 
where $f_1|_{H^1}$ is the identity $H^1 \xto{=} \Zh^1 = \Hom_{\cont}(G,\Fp)$, 
and each $f_i$ is of degree $1-i$, 
such that 
$m_3|_{A_V^{\otimes3}}=\cdots=m_n|_{A_V^{\otimes n}} = 0$. 
For $1\leq r\leq n$, we set 
$\beta_r \coloneqq \varepsilon_r f_r|_{K_r(V)} \colon K_r(V) \to \Ch^1$.  
In particular, $\beta_1 = \id_V$. 
Since $\varepsilon_r^2=1$, the definition is equivalent to $f_r|_{K_r(V)} = \varepsilon_r \beta_r$. 
Now we restrict the $r$-th $A_\infty$-morphism identity to $K_r(V)$. 
Every source term containing an insertion of $m_2$ vanishes since neighbouring pairs in
$K_r(V)$ lie in $R_V$. 
Every source term involving one of $m_3, \ldots,m_r$ 
vanishes by the assumption in (a). % \eqref{eq:vanishing-through-n}. %the assumed triviality of the operations on $A_V$. 
The remaining target multiplication terms combine to $\Omega_r(\beta)$. 
Hence 
$\delta\bigl(f_r|_{K_r(V)}\bigr)  = - \Omega_r(\beta) = - \varepsilon_r S_r(\beta)$, 
and thus $\delta\beta_r = - S_r(\beta)$ for $2 \le r \le n$. 
Thus, $\beta = (\beta_1,\ldots, \beta_n)$ is a $V_{n+1}$-good family. 
By Lemma \ref{lem:truncated-twisting-cochain}, 
this implies (b). 

%%%%%%%%%%%

Now we assume that (b) holds. 
By Lemma \ref{lem:truncated-twisting-cochain}, 
the cochains $b_1,\dots,b_n$ determine a $V_{n+1}$-good family $\beta = (\beta_1,\ldots, \beta_n)$. 
By Proposition \ref{prop:defining-system-comparison} applied with $r = n+1$, 
$\Hb$ is $V_{n+1}^{\beta}$-nice, which proves (a). 
Hence (b) implies (a). 
\end{proof}

%%%%%%%%%%%%%%%%%
%
%%%%%%%%%%%%%%%%%

\begin{corollary}\label{cor:weak-embedding} 
Let $G$ be a profinite group such that $H^i(G,\F_p) = 0$ for $i \ge 3$, 
and let $n \ge 3$. 
Then $\Cb(G,\Fp)$ is weakly $A_n$-formal 
if and only if every finite embedding problem $\mathbf E^{(n)}_V$ is solvable. 
Moreover, $\Cb(G,\Fp)$ is $A_n$-formal 
if and only if the embedding problem $\mathbf E^{(n)}_{H^1}$ for $V = H^1$ is solvable. 
\end{corollary}

\begin{proof}
This follows from Theorem \ref{thm:finite-lifting} and the definition of weak and genuine $A_n$-formality, respectively. 
\end{proof}

%%%%%%%%%%%%%%%%
%
%%%%%%%%%%%%%%%%

%
Corollary \ref{cor:weak-embedding} relates weak $A_n$-formality to the existence of 
a solution to every finite embedding problem, 
whereas genuine $A_n$-formality corresponds to the existence of a solution of the global embedding problem. 
Now we show how we can glue the finite data together. 
%%%%%%
%
Recall that we write $\mathcal F(H^1)$ for the directed poset of finite-dimensional 
vector subspaces of $H^1$.  
For $V\subset W$ in $\mathcal F(H^1)$, 
let
$r_{WV} \colon \UA_n(W)\twoheadrightarrow\UA_n(V)$  
be the homomorphism induced by restriction on Pontryagin duals.  
The surjection $r_{WV}$ 
sends a solution of $\mathbf E^{(n)}_W$ to a solution of 
$\mathbf E^{(n)}_V$. 
Let $\Sol(\mathbf E^{(n)}_V)$ 
denote the set of   solutions of $\mathbf E^{(n)}_V$. 
We then get transition maps 
\[
\pi^{(n)}_{WV} \colon \Sol(\mathbf E^{(n)}_W)\longrightarrow\Sol(\mathbf E^{(n)}_V),
~ \rho_n \longmapsto r_{WV}\circ\rho_n.
\]
These maps make the solution sets into a projective system indexed by
$\mathcal F(H^1)$. 
%
%%%%%%%%
%
%
Now we can prove the following important technical result.  

\begin{theorem}
\label{thm:compatible-defining-systems}
Let $G$ be a profinite group such that $H^i(G,\F_p) = 0$ for $i \ge 3$, 
and let $n \ge 3$. 
Then the following are equivalent:
\begin{enumerate}
\item[\textnormal{(a)}] $\Cb(G,\F_p)$ is $A_n$-formal; 
\item[\textnormal{(b)}] the finite embedding problems $\mathbf E^{(n)}_V$ admit a compatible
family of solutions, i.e.,  
\[
\varprojlim_{V\in\mathcal F(H^1)}\Sol(\mathbf E^{(n)}_V)\ne\varnothing.
\]
\end{enumerate}
\end{theorem}

\begin{proof}
Suppose first that $\Cb(G,\Fp)$ is $A_n$-formal. 
By Corollary \ref{cor:weak-embedding}, the
global embedding problem
$G \xrightarrow{\chi_{H^1}} (H^1)^\vee \twoheadleftarrow \UA_n(H^1)$  
has a solution $\rho \colon G \to \UA_n(H^1)$. 
For every $V \in \mathcal F(H^1)$, 
let $r_V \colon \UA_n(H^1) \to \UA_n(V)$ be restriction and put
$\rho_V \coloneqq r_V\circ\rho$. 
Then $\rho_V$ is a solution of
$\mathbf E^{(n)}_V$, and for $V\subset W$ one has
$r_{WV}\circ\rho_W=\rho_V$.  
Thus $(\rho_V)_V$ is a point of
$\varprojlim_V \Sol(\mathbf E^{(n)}_V)$. 
This proves that (a) implies (b). 
%
%%%%%%%

Now we assume the projective system is nonempty and let  
\[
(\rho_V)_V \in \varprojlim_V \Sol(\mathbf E^{(n)}_V).
\]
For each $g\in G$, the elements $(\rho_V(g))_V$ form a compatible family.
By \eqref{eq:U-limit}, they determine a unique element $\rho(g)\in \UA_n(H^1)$.
This defines a map $\rho \colon G\to \UA_n(H^1)$.
For $g,h\in G$ and every $V \in \mathcal F(H^1)$, 
we have 
\[
r_V\bigl(\rho(gh)\bigr)=\rho_V(gh)
=\rho_V(g)\rho_V(h)
=r_V\bigl(\rho(g)\rho(h)\bigr).
\]
The projections in \eqref{eq:U-limit} separate points, so $\rho$ is a group homomorphism. 
It is continuous since \eqref{eq:U-limit} is a homeomorphism 
and every component $r_V\circ\rho=\rho_V$ is continuous. 
Let $q \colon \UA_n(H^1) \twoheadrightarrow (H^1)^\vee$ be the canonical quotient. 
For every $V \in \mathcal F(H^1)$, %For every finite-dimensional $V$, 
the $V^\vee$-component of $q\circ\rho$ is
$G\xrightarrow{\rho_V} \UA_n(V) \twoheadrightarrow V^\vee$, 
which equals $\chi_V$. 
Since the $V$-component of $\chi_{H^1}$ is $\chi_V$, 
it follows that
$q\circ\rho=\chi_{H^1}$. 
Thus, $\rho$ solves the embedding problem $\mathbf E^{(n)}_{H^1}$. 
Corollary \ref{cor:weak-embedding} 
now implies that $\Cb(G,\Fp)$ is $A_n$-formal. 
This proves that (b) implies (a). 
\end{proof}

%%%%%%%%%%%%%%%%%%
%
%%%%%%%%%%%%%%%%%%

\section{Failure of weak \texorpdfstring{$A_3$}{A3}-formality at the prime two} 
\label{sec:weak-counterexamples}

Before we apply the results of Section \ref{sec:finite-embedding-comparison} for odd primes, 
we provide some explicit examples of fields for which weak $A_3$-formality fails at the prime $p=2$. 

\begin{proposition}
\label{prop:nonA3-examples}
The differential graded algebra $\Cb(G_F,\F_2)$ is not weakly $A_3$-formal % in degree one 
in each of the following cases: 
%\Gereon{are (ii) or (iii) new? clarify}
\begin{enumerate}
\item[(i)] $F$ is a finite extension of $\Q$ of odd degree, 
\item[(ii)] $F=\mathbb F_{\ell}(t)$, where $\ell \equiv 3 \pmod 8$,
\item[(iii)] $F=k(s,u,w)$, where $k$ is any field of characteristic
different from $2$.
\end{enumerate}
\end{proposition}

\begin{remark}
We note that, to the best of our knowledge, assertion (ii) in Proposition \ref{prop:nonA3-examples} is new. 
Assertion (i) follows from the Harpaz--Wittenberg example together with \cite[Corollary 3.9]{MS1},   
as we explain in the proof. 
Assertion (iii) is new in the presented form, 
although it is the $p=2$ specialisation of the construction in \cite[Theorem~1.3]{MSNF} 
as we explain in the proof of  Proposition \ref{prop:nonA3-examples} below. 
However, the three-variable rational-function-field formulation does not appear to be stated explicitly in \cite{MSNF}. 
\end{remark}

%%%%%%%%%%%%%%%%%%

\begin{remark}\label{rem:p2_obstruction_to_weak_A3}
For the proof of Proposition \ref{prop:nonA3-examples}, 
we use the following consequence of weak $A_3$-formality. 
If the minimal structure can be chosen with $m_3$
vanishing on the span of four classes $a,b,c,d\in H^1$, and if
$a \cup b = b \cup c = c \cup d=0$,  
then the fourfold Massey product $\langle a,b,c,d\rangle$ is defined,
see  for example \cite[proof of Proposition 2.17]{PQodd}.  % change to Prop 2.17 after {PQodd} is updated on the arxiv.
Thus, a single four-tuple violating this property obstructs weak $A_3$-formality. 
\end{remark} 

%%%%%%%%%%%%

\begin{notn}
We use the following notation from \cite[page 3502]{MS1}.
Let $F$ be a field of characteristic different from $2$. 
For $a \in F^{\times}$, we let $F_a$ be the quadratic \'etale $F$-algebra $F[x_a]/(x_a^2-a)$, 
and write $N_a$ for the image of the norm map $F_a^\times \to F^\times$. 
For  $b,c\in F^\times$, 
we write $F_{b,c} \coloneqq F[x_b,x_c]/(x_b^2-b,x_c^2-c)$  
for the associated degree-four finite \'etale $F$-algebra, and
$N_{b,c}  \coloneqq \Nm_{F_{b,c}/F}(F_{b,c}^{\times})\subseteq F^\times$ 
for its norm group.  
We identify $F^\times/F^{\times 2}$ with
$H^1(G_F,\F_2)$ and write $(b,c)$ for the corresponding quaternion symbol in $H^2(G_F,\F_2)\cong \Br(F)[2]$. 
\end{notn}

%%%%%%%%%%%%%%%%

\begin{lemma}\label{lemma:odd-degree}
Let $F$ be a field of characteristic different from $2$, and suppose
that the neighbouring symbols of
$\langle bc,b,c,bc\rangle$ vanish. 
Then: 
\begin{enumerate}
\item[\textnormal{(i)}] The Massey product $\langle bc,b,c,bc\rangle$ is defined if and only if $-1\in N_{b,c}$. 
\item[\textnormal{(ii)}] If $\langle bc,b,c,bc\rangle$ is not defined over $F$ and $E/F$ is a finite extension of odd degree, 
then $\langle bc,b,c,bc\rangle$ remains non-defined after base change to $E$.
\end{enumerate}
\end{lemma}

\begin{proof}
Assertion (i) is a special case of \cite[Theorem 1.4]{MS1}. %
For (ii), the first assertion gives  
$-1\notin \Nm_{F_{b,c}/F}(F_{b,c}^{\times})$. 
Suppose that $-1$ becomes a norm after extending scalars to $E$.  
Then there is an element 
$z\in (F_{b,c}\otimes_F E)^\times$  
with $\Nm_{F_{b,c}\otimes_F E/E}(z)=-1$. 
Compatibility and transitivity of norms 
for finite \'etale  algebras give 
\[
\Nm_{F_{b,c}/F}
 \bigl(\Nm_{F_{b,c}\otimes_F E/F_{b,c}}(z)\bigr)  =\Nm_{E/F}(-1)   =(-1)^{[E:F]} =-1. 
 \] 
The neighbouring symbols remain zero after base change 
since restriction preserves their vanishing.
This contradicts the assumption and proves (ii). 
\end{proof}

%%%%%%%%%%%%

\begin{proof}[Proof of Proposition \ref{prop:nonA3-examples}]
For (i), the example of Harpaz--Wittenberg in \cite[Example A.15]{GMT}, 
which is also recovered in \cite[Proposition 5.3]{MS1}, 
shows $\angles{bc,b,c,bc}$ for $b=2$ and $c=17$ 
is not defined over $\Q$ 
while the neighbouring cup-products vanish. 
Let $F/\Q$ be a finite extension of odd degree. 
By Lemma \ref{lemma:odd-degree}, $\angles{34,2,17,34}$ is then not defined over $F$ 
while the neighbouring cup-products vanish. 
This implies that $\Cb(G_F,\F_2)$ is not weakly $A_3$-formal. 

%%%%%%%%%%%%%

For (ii), we assume $\ell \equiv 3 \pmod8$ and write $F = \F_{\ell}(t)$. 
We set $b=2$, $c=t^2+1$, and $a=d=bc=2(t^2+1)$. 
Since both $-1$ and $2$ are non-squares in $\F_{\ell}$, 
the element $-2$ is a square.  
We choose $\alpha \in \F_{\ell}$ with  $\alpha^2=-2$.  
Then $(t^2+1)\cdot1^2 + 2 (\alpha t/2)^2=1$,
and hence $(2,t^2+1)=0$.  
Moreover, 
$2\cdot1^2+2\cdot1^2=2^2$, 
and $(t^2+1)\cdot1^2-1\cdot1^2=t^2$, 
so $(2,2)=0$ and $(t^2+1,t^2+1)=0$. 
This implies $(a,b)=(b,c)=(c,d)=0$. 
We claim that
$-1\notin N_{2,t^2+1}$. 
Set $L  \coloneqq F(\sqrt2,\sqrt{t^2+1})$.  
At every place $v$ of $F$, the local degree of $L/F$ is $1$ or $2$.  
Indeed, if $2$ is a square in $F_v$ there
is nothing to prove.  
Otherwise, $F_v(\sqrt2)/F_v$ is the unramified quadratic extension.  
Since $(2,t^2+1)_v=0$, the element $t^2+1$ is a norm
from that extension and therefore has even valuation.  
Its local square  class is consequently either trivial or the class of the non-square unit $2$, 
so adjoining both square roots still gives degree at most $2$.
%
%%%%%%
%
We may therefore use the Serre--Tate multinorm criterion in the form used in the proof of 
\cite[Lemma 5.2 on page 3530]{MS1}.  
Let $S_1$ be the set of places of $F$ which split
in $F(\sqrt2)$ and define
$\omega(q) \coloneqq \prod_{v\in S_1}(t^2+1,q)_v$. 
Then, with $c= t^2+1$, 
\begin{equation}\label{eq:ST-multinorm}
q\in N_2N_cN_{2c}
~ \text{if and only if} ~ \omega(q)=1. 
\end{equation}
We apply this to $q=t+1$. 
%
%%%%
%
The quadratic extension $F(\sqrt2)/F$ is the constant extension
$\F_{\ell^2}(t)/\F_{\ell}(t)$, so a finite place splits precisely
when its degree is even.  
The only place in $S_1$ at which
$(c,t+1)_v$ can be non-trivial is the degree-two place
$v_c=(t^2+1)$.  
Its residue field is $\F_{\ell^2} = \F_{\ell}(i)$, $i^2=-1$, and the
residue of $t+1$ is $1+i$.  
Now
$N_{\F_{\ell^2}/\F_{\ell}}(1+i)=2$.   
For $z\in\F_{\ell^2}^{\times}$, $z$ is a square if and only if its
norm to $\F_{\ell}$ is a square.  
Since $2$ is a non-square in
$\F_{\ell}$, the element $1+i$ is a non-square in $\F_{\ell^2}$.
Thus
$(c,t+1)_{v_c}=-1$, 
and therefore
$\omega(t+1)=-1$.  
By \eqref{eq:ST-multinorm},
$t+1\notin N_2N_cN_{2c}$.  
Now \cite[Lemma A.3(2)]{MS1} implies $(t+1)^2\notin N_{2,c}$. 
On the other hand, there is a very simple explicit norm giving the opposite sign. 
We write
$r = \sqrt2$, $s=\sqrt{t^2+1}$, 
and set
$\rho \coloneqq \frac{t+1+s-r}{r}\in L^\times$.  
Multiplication over the four choices of signs of $r$ and $s$ gives
\[
\prod_{\epsilon,\eta=\pm1}
(t+1+\eta s-\epsilon r)=-4(t+1)^2,
\]
whereas $N_{L/F}(r)=4$.  
Hence $N_{2,c}(\rho)=-(t+1)^2$. 
If $-1$ belonged to $N_{2,c}$, dividing $N_{2,c}(\rho)=-(t+1)^2$ 
by such a norm would imply that
$(t+1)^2\in N_{2,c}$, contradicting $(t+1)^2\notin N_{2,c}$. 
Thus,  $-1\notin N_{2,c}$, and (i) in Lemma \ref{lemma:odd-degree} shows that
\[
\langle 2(t^2+1),2,t^2+1,2(t^2+1)\rangle
\]
is not defined.  This proves (ii).

%%%%%%%%%%%%

For (iii), in the proof of \cite[Theorem 1.3]{MSNF},   
Merkurjev--Scavia construct the non-defined fourfold product
$\langle1-s,s,t,1-t\rangle$ over the function field of the
Severi--Brauer conic of $(s,t)$. 
%
%\cite[Theorem 1.3]{MSNF}.  
%
On the affine chart $t=w^2-su^2$,  
this function field is the fraction field 
$\mathrm{Frac}(k(s,t)[u,w]/(t-w^2+su^2)) \cong k(s,u,w)$, 
giving the explicit Massey product 
\[
\left\langle
1-s,\ s,\ w^2-su^2,\ 1-w^2+su^2
\right\rangle
\]
which is not defined. 
\end{proof}

%%%%%%%%%%%%%%%%%%%
%%
%%%%%%%%%%%%%%%

\section{The \texorpdfstring{$A_n$}{An}-local--global principle for number fields}\label{sec:numberfields-local-global}

Now we apply the results of Section \ref{sec:finite-embedding-comparison} 
to the Galois cochains of a number field and odd primes $p$. 
Let $K$ be a number field and let $v$ be a place of $K$. 
Let $G_K$ and $G_{K_v}$ denote the absolute Galois groups of $K$ and $K_v$ respectively. 
The choice of an embedding $\overline{K} \into \overline{K}_v$ induces 
an injective homomorphism $\iota_v \colon G_{K_v} \into G_K$ 
whose conjugacy class is independent of the initial choice. 
Let $M$ be a finite $G_K$-module. 
% 
%For a place $v$ of $K$, w
We write 
$\loc_{v} \colon \Hb(G_K,M) \to \Hb(G_{K_v},M)$ for the localisation homomorphism 
induced by restriction along $G_{K_v} \into G_K$. 
For $c \in H^i(G_K,\F_p)$, we also often write $c_v = \loc_v(c)$ to shorten the notation. 
Let  $V \subset H^1(G_{K},\F_p)$ be a finite-dimensional vector subspace, 
and let $W \coloneqq \loc_v(V) \subseteq H^1(G_{K_v},\F_p)$ be the image of $V$ under the localisation map $\loc_v$. 
Since localisation sends $R_V$ into $R_W$, it sends $K_m(V)$ into $K_m(W)$ for every $m$. 
Dualising gives compatible homomorphisms $P_m(W)\to P_m(V)$ and hence a homomorphism $\UA_n(W) \to \UA_n(V)$.
Then we get a diagram of embedding problems 
\begin{align}\label{eq:diagram_embedding_problems}
\xymatrix{
\UA_n(W) \ar[d] \ar@{->>}[r] & W^\vee \ar[d] & \ar[l]_{\chi_{W}} \ar[dl]^{\chi_{V,v}} G_{K_v} \ar[d] \\
\UA_n(V) \ar@{->>}[r] & V^\vee & \ar[l]^{\chi_{V}} G_{K}
}
\end{align}
in which the squares and triangles commute. 

\begin{lemma}
\label{lem:local-plans-general}
Let $v$ be a place of $K$, and assume that $H^i(G_{K_v},\F_p) =0 $ for $ i \ge 3$. 
Assume that $\Cb(G_{K_v},\Fp)$ is $A_n$-formal. 
Then, for every finite-dimensional vector subspace $V\subset H^1(G_K,\Fp)$, 
the localised embedding problem
\[
\mathbf E^{(n)}_{V,v} \colon G_{K_v} \xrightarrow{\chi_{V,v}} V^\vee \twoheadleftarrow \UA_n(V)
\]
has a solution.
\end{lemma}

\begin{proof} 
Let $V\subset H^1(G_K,\Fp)$ be a finite-dimensional subspace, and let $W$ be its image under $\loc_v$. 
By Theorem \ref{thm:finite-lifting},  
the embedding problem 
$\mathbf E^{(n)}_W \colon G_{K_v}\xrightarrow{\chi_{W}} W^\vee  \twoheadleftarrow \Un(W)$ is solvable. 
We then see from diagram \eqref{eq:diagram_embedding_problems}, that this implies that $\mathbf E^{(n)}_{V,v}$ is solvable. 
\end{proof}

%%%%%%%%%%%%%%%%%

We now assume that $p$ is odd. We assume that  $K$ is a number field with $\mu_p\not\subset K$. 
Our next goal is to show a local--global principle for $A_n$-formality. 

\begin{proposition}
\label{prop:local-to-global}
Let $K$ be a number field, let $p$ be odd, assume
$\mu_p\not\subset K$, 
and let $n\ge3$.  
If $\Cb(G_{K_v},\Fp)$ is $A_n$-formal for every completion $K_v$ of $K$, 
then $\Cb(G_K,\Fp)$ is weakly $A_n$-formal. 
\end{proposition}

\begin{proof}
Let $V\subset H^1(G_K,\Fp)$ be a finite-dimensional vector subspace. 
Let $F_V/K$ 
be the elementary abelian $p$-extension determined by $\chi_V \colon G_K \to V^\vee$ 
such that 
$\Gal(F_V/K)\cong V^\vee$.  
Let $S$ be the finite set of finite primes ramified in $F_V/K$.  
For each $v\in S$, 
Lemma~\ref{lem:local-plans-general} gives a continuous  homomorphism %local plan 
$\rho_v \colon G_{K_v} \to \UA_n(V)$ 
lifting $\chi_{V,v}$.  
The group $\UA_n(V)$ is a finite $p$-group and
$N_{n,V}=\ker\bigl(\UA_n(V)\to V^\vee \bigr)$ 
is a normal subgroup.  
We may therefore apply the general lifting theorem of Maire--Min\'a\v c--Ramakrishna--T\^an \cite[Theorem 3.2]{MMRT} 
with $U=\UA_n(V)$, $N=N_{n,V}$, and $U/N \cong V^\vee$. 
It produces a finite Galois extension $L/K$ containing $F_V$ with
$\Gal(L/K)\simeq\UA_n(V)$ 
such that $G_K\to\Gal(L/K)$ is a solution of $\mathbf E^{(n)}_V$. 
Since $V$ was arbitrary, 
Corollary \ref{cor:weak-embedding} implies  weak $A_n$-formality for $\Cb(G_{K},\Fp)$. 
\end{proof}

%%%%%%%%%%%%%%%
%
For the converse, we need to lift embedding problems 
while preserving the cup-product relations. 
We now show that this can be achieved under our cyclotomic hypothesis. % this can be done.
For a number field $K$ and a finite $G_K$-module $M$, 
we denote by 
$\Sha^i(K,M) =   \ker (H^i(K,M)\to\prod_v H^i(K_v,M) )$  
the $i$-th Tate--Shafarevich group, where the product is taken over all places of $K$. 

\begin{lemma}
\label{lem:Sha2}
Let $K$ be a number field, $p$ be odd and $\mu_p\not\subset K$.  
Then $\Sha^2(K,\Fp) = 0$. 
\end{lemma}

\begin{proof}
Poitou--Tate duality, see e.\,g.\,\cite[Theorem 8.6.7]{NSW}, 
identifies $\Sha^2(K,\Fp)$ with the dual of $\Sha^1(K,\mu_p)$.  
We recall why the latter vanishes.  % $K'=K(\mu_p)$.  
If a class in $H^1(G_K,\mu_p)$ is locally trivial, 
its restriction to $K(\mu_p)$ is a homomorphism to the now constant group $\mu_p$ which is trivial on every decomposition group.  
A nonzero such homomorphism determines a nontrivial cyclic extension of $K(\mu_p)$, 
and Chebotarev supplies a prime with nontrivial Frobenius, contradicting local triviality.  
Thus, the class restricts to zero over $K(\mu_p)$.  
Inflation--restriction and the nontrivial prime-to-$p$ cyclotomic action then force the original class to vanish.
\end{proof}

%%%%%%%%%%%%%

Next, we need to show that we can lift a subspace in local cohomology to global cohomology 
such that the local and global kernel of the cup-product are preserved. 
To do so, 
we use an approximation result from the work of Maire--Min\'a\v c--Ramakrishna--T\^an \cite[Proposition 3.1]{MMRT}
which says that, given finitely many local classes and a finite $p$-extension $F/K$, 
one can realise the local data by a global class after allowing one new tame ramified prime, 
and this new prime can be chosen to split completely in $F/K$.

\begin{proposition}\label{prop:relation-approx} 
Let $K$ be a number field, let $p$ be odd, assume $\mu_p\not\subset K$, and let $v$ be a finite place.  
Let $W\subset H^1(G_{K_v},\Fp)$ 
be finite-dimensional.  
Then there exists a finite-dimensional subspace
$V\subset H^1(G_K,\Fp)$ 
such that localisation induces an isomorphism 
$\loc_v \colon  V \xrightarrow{\cong} W$ 
and, under the tensor-square of this isomorphism,
$R_V \xrightarrow{\cong} R_W$. 
For every $n \ge 2$, the induced natural map 
$\UA_n(W)\xrightarrow{\cong}\UA_n(V)$ is then an isomorphism. 
\end{proposition}

\begin{proof}
Choose a basis $w_1,\dots,w_d$ of $W$.  
We construct global classes $a_1,\dots,a_d$ inductively, 
together with distinct tame primes $q_1,\dots,q_d$ such that
\begin{enumerate}[label=\textnormal{(\roman*)}]
\item $(a_i)_v = w_i$,
\item $a_i$ is unramified outside $\{v,q_1,\dots,q_i\}$,
\item $(a_i)_{q_j}=0$ for $j<i$,
\item $q_i$ splits completely in the elementary $p$-extension determined by $a_1,\dots,a_{i-1}$.
\end{enumerate}
Suppose $a_1,\dots,a_{i-1}$ have been constructed.  
Then we apply \cite[Proposition 3.1]{MMRT} to the finite set 
$X_i=\{v,q_1,\dots,q_{i-1}\}$, 
with prescribed local value $w_i$ at $v$ and $0$ at each $q_j$, 
and with the finite $p$-extension determined by $a_1,\dots,a_{i-1}$.  
By \cite[Proposition 3.1]{MMRT}, we can find $a_i$ and a new tame prime $q_i$ satisfying (i)--(iv). 
We set $V \coloneqq \langle a_1,\dots,a_d\rangle \subset H^1(G_K,\F_p)$.  
Localisation at $v$ is an isomorphism $\phi \coloneqq \loc_v|_{V} \colon V \xrightarrow{\cong} W$ 
since the $w_i$ are linearly independent. 
We have the inclusion $(\phi \otimes \phi)(R_V) \subseteq R_W$.  
It remains to show that this is an equality. 
We claim that for every pair $i,j$, 
the cup-product $a_i \cup a_j$ is locally zero away from $v$.  
At a prime $q_s$, at most one of the localisations $(a_i)_{q_s}$ is nonzero, 
since earlier classes vanish there because $q_s$ splits in the extension they determine,  
while later classes were constructed to vanish at $q_s$. 
Since $p$ is odd, 
$a_s \cup a_s = 0$ 
by graded commutativity.  
At every other finite place away from $v$, 
all $a_i$ are unramified, and the cup-product of two unramified $p$-characters is zero 
since the maximal unramified pro-$p$ quotient of the corresponding Galois group  has cohomological dimension one.  
Now let 
$r \in V\otimes V$. 
By the above arguments, 
we have $\loc_q(\mu(r)) = 0$ for all primes $q\ne v$. 
At the archimedean places, the localisation is also zero, since $p$ is odd. 
Thus, if $\loc_v(\mu(r)) = 0$, 
then $\mu(r)$ lies in $\Sha^2(K,\Fp)$.  
Lemma \ref{lem:Sha2} then implies $\mu(r) = 0$ in $H^2(G_K,\Fp)$. 
This shows $(\phi \otimes \phi)^{-1}(R_W) \subseteq R_V$. 
Hence, $(\phi \otimes \phi)|_{R_V}  \colon R_V \xrightarrow{\cong} R_W$ is an isomorphism, 
and the global and local kernels agree. 
The assertion for $\UA_n$ follows functorially from the definition.
\end{proof}

%%%%%%%%%%%%%%%

We can now prove the local--global principle for weak $A_n$-formality. 

\begin{theorem}\label{thm:LG-equivalence}
Let $K$ be a number field, let $p$ be odd, assume
$\mu_p\not\subset K$, 
and let $n\ge3$.  
Then the following are equivalent:
\begin{enumerate}[label=\textnormal{(\roman*)}]
\item $\Cb(G_K,\Fp)$ is weakly $A_n$-formal, 
\item for every place $v$ of $K$, 
$\Cb(G_{K_v},\Fp)$ is $A_n$-formal.
\end{enumerate}
\end{theorem}

\begin{proof}
Proposition \ref{prop:local-to-global} shows that (ii) implies (i). 
Now we prove that (i) implies (ii). 
At an archimedean place, 
$H^i(G_{K_v},\F_p)=0$ for $i\ge 1$ since $p$ is odd. 
Hence we can assume that $v$ is finite. 
Suppose that $\Cb(G_{K_v},\Fp)$ is not $A_n$-formal.  
Since $v$ is finite, $H^1(G_{K_v},\F_p)$ is finite-dimensional, 
and since $p$ is odd, $H^i(G_{K_v},\F_p) = 0$ for $i \ge 3$.  
By Corollary \ref{cor:weak-embedding} applied with $W=H^1(G_{K_v},\F_p)$, 
the embedding problem 
\[
\mathbf E^{(n)}_W \colon  G_{K_v} \longrightarrow W^\vee \twoheadleftarrow \UA_n(W)  
\]
has no solution.
Then we apply Proposition \ref{prop:relation-approx} with $W=H^1(G_{K_v},\F_p)$
to obtain a finite-dimensional vector subspace $V\subset H^1(G_K,\Fp)$ with
$V \cong W$, $R_V \cong R_W$, and $\UA_n(V) \cong \UA_n(W)$.  
If $\Cb(G_K,\Fp)$ is weakly $A_n$-formal, 
the embedding problem $\mathbf E^{(n)}_V$ has a solution by Corollary \ref{cor:weak-embedding}.  
Restricting the global solution to $G_{K_v}$ 
gives a solution with target $\UA_n(V)$. 
Composing with the inverse of the natural isomorphism $\UA_n(W) \xrightarrow{\cong} \UA_n(V)$ 
from Proposition \ref{prop:relation-approx} gives a solution of $\mathbf E^{(n)}_W$, a contradiction.  
Hence (i) implies (ii).
\end{proof}

\begin{remark}
We note that for the proof of Theorem \ref{thm:LG-equivalence} 
it is crucial that we can lift cohomology classes 
preserving their cup-product relations using Proposition \ref{prop:relation-approx}. 
Lifting a basis of $H^1(G_{K_v},\Fp)$ is not enough, 
since  the global cup-product kernel must coincide with the local one 
to lift the embedding problem from $V$ to $W$. 
\end{remark}

%%%%%%%%%%%

\begin{remark}\label{rem:failure_at_p=2_local-global}
We now explain why our argument for the local-global principle does not apply when $p=2$. 
First, the proof of Proposition \ref{prop:local-to-global} 
uses the lifting theorem of Maire--Min\'a\v c--Ramakrishna--T\^an \cite[Theorem 3.2]{MMRT}, 
and the assumption $p$ odd and $\mu_p \not\subset K$ are crucial for that theorem. 
Second, the proof of Proposition \ref{prop:relation-approx} uses 
$a \cup a = 0$, which is automatic for odd $p$ but false in general in mod-$2$ Galois cohomology. 
For example, $a \cup a = 0$ fails for certain classes of pro-$2$ Demushkin groups, see \cite{PQtwo}. 
Finally, in Sections \ref{sec:weak_vs_genuine_An_general} and \ref{sec:finite-embedding-comparison} 
we assumed $H^i(G,\F_p) = 0$ for $i \ge 3$. 
This assumption is implied by $\cd_p(G) \le 2$. 
This assumption may not be satisfied at $p=2$. 
For example, we have $\cd_2(\Q) >2$. 
Thus, our proof of Theorem \ref{thm:LG-equivalence} 
does not extend to $p=2$.  
Therefore, the fact that $\Cb(G_{K_v},\F_2)$ is $A_3$-formal for every place $v$ by \cite[Theorem 1.2]{PQtwo} 
does not imply that $\Cb(G_K,\F_2)$ is weakly $A_3$-formal.  
In fact, for $n=3$ and $p=2$, 
the analogous local-to-global implication is false in general by Proposition \ref{prop:nonA3-examples}. 
\end{remark}

%%%%%%%%%%%%%%%

\section{Weak \texorpdfstring{$A_3$}{A3}-formality for number fields}\label{sec:number-fields-weak}

From now on, we consider the case $n=3$. 
For the cases where number fields are weakly $A_3$-formal, we provide two proofs. 
%
%%%%%%%%%%%
%
Both proofs use the determination of $A_3$-formality for pro-$p$ Demushkin groups via 
concrete computations in \cite{PQodd}. 
The first proof is, in addition, based on the local--global principle for weak $A_n$-formality.  

\begin{theorem}
\label{thm:numberfield-absolute-trim}
Let $K$ be a number field, let $p$ be an odd prime number, and assume $\mu_p\not\subset K$.  
For $p=3$, we assume that, for every place $v$ of $K$, 
if  $\mu_3 \subset K_v$ then $\mu_9 \subset K_v$. 
Then $\Cb(G_K,\F_p)$ is weakly $A_3$-formal.
\end{theorem}

\begin{proof}[First proof of Theorem \ref{thm:numberfield-absolute-trim}]
By Proposition \ref{prop:local-to-global}, 
it suffices to show that $\Cb(G_{K_v},\F_p)$ is $A_3$-formal for every place $v$ of $K$. 
At an archimedean place, 
$H^i(G_{K_v},\F_p)=0$ for $i\ge 1$ 
since $p$ is odd.  
%\Gereon{add an argument here!} 
%
Suppose that $K_v$ is nonarchimedean.  
If $\mu_p\not\subset K_v$, local Tate duality gives
$H^2(G_{K_v},\F_p)=0$.   
In both cases, the obstruction target vanishes. 
Hence $\ob_{3,W} = 0$ for every subspace $W \subseteq H^1(G_{K_v},\F_p)$, and in particular  $\ob_{3,H^1} = 0$. 
Proposition \ref{prop:global-two-step} therefore implies that the local cochain algebra is $A_3$-formal. 
If $\mu_p\subset K_v$, the maximal pro-$p$ quotient $G_{K_v}(p)$ of $G_{K_v}$ is a Demushkin group 
by the work of Demushkin \cite{Demushkin_odd} (see also for example \cite[Proposition 7.5.9]{NSW}). 
By the assumption on $K$, the $q$-invariant  of $G_{K_v}(p)$ is 
$q = p^f$ with $f\ge 1$ when $p\ge5$ and $f \ge 2$ when $p=3$, 
and hence $q \ne 3$.   
By \cite[Theorem 1.2]{PQodd}, 
we deduce that $\Cb(G_{K_v}(p),\F_p)$ is $A_3$-formal. 
Since the quotient map $G_{K_v} \to G_{K_v}(p)$ induces a quasi-isomorphism 
$\Cb(G_{K_v}(p),\F_p) \to \Cb(G_{K_v},\F_p)$ (see for example \cite[Proposition 2.2 on page 309]{CMOSU} or \cite[Remark 1.4]{MPQT}),  
this implies that 
$\Cb(G_{K_v},\F_p)$ is $A_3$-formal. 
\end{proof}

%%%%%%%%%%%

For $p=3$, we now show how we can characterise fields which satisfy the assumption of Theorem \ref{thm:numberfield-absolute-trim}.  
\begin{lemma}\label{lemma:criterion_p=3}
Let $K$ be a number field with $\mu_3\not\subset K$.  
The following
conditions are equivalent:
\begin{enumerate}
\item There exists a finite place $v$ of $K$ such that $\mu_3\subset K_v$ and $\mu_9\not\subset K_v$. 
\item One has $K(\mu_3)\neq K(\mu_9)$. 
\end{enumerate}
If these conditions hold, then there are infinitely many such places $v$. 
\end{lemma}

\begin{proof}
Let $\zeta_n$ denote a primitive $n$-th root of unity.  
Since
$\zeta_3=\zeta_9^3$, we have $K(\mu_3) \subseteq K(\mu_9)$. 
First suppose that $K(\mu_3) = K(\mu_9)$.   
If $\mu_3\subset K_v$, then the inclusion
of $K$ into $K_v$, together with a primitive third root of unity in $K_v$,
gives a $K$-embedding
$K(\mu_3)\into K_v$.  
Since $K(\mu_3) = K(\mu_9)$, the image of this embedding contains all $9$-th
roots of unity.  
Thus, $\mu_9\subset K_v$.  

Now we assume  $K(\mu_3) \neq K(\mu_9)$.  
Since $X^2+X+1$ has no root in $K$, it is irreducible, 
and we have $[K(\mu_3):K]=2$.  
Moreover, $K(\mu_9)/K(\mu_3)$ is a nontrivial extension of degree dividing $3$, 
and hence $[K(\mu_9):K(\mu_3)] = 3$.  
Choose a nonidentity element
$\sigma\in\Gal(K(\mu_9)/K(\mu_3))$.  
The element $\sigma$ acts trivially on $K(\mu_3)$, 
and hence on $\mu_3$, 
but acts nontrivially on $K(\mu_9)$, and hence on $\mu_9$. 
The extension $K(\mu_9)/K$ is Galois and abelian 
since it is the compositum of $K$ with
the abelian cyclotomic extension $\Q(\zeta_9)/\Q$. 
By Chebotarev's density theorem (see for example \cite[Theorem 9.1.3]{NSW}), 
after excluding the finitely many primes above $3$, there are infinitely many finite places $v$, 
unramified in $K(\mu_9)/K$, with Frobenius $\sigma$.

Because the restriction of $\sigma$ to $K(\mu_3)$ is the identity, 
the Frobenius of $v$ in $K(\mu_3)/K$ is trivial.  
Therefore, $v$ splits completely in $K(\mu_3)/K$.  
For any place $w$ of $K(\mu_3)$ above $v$, this gives
$K(\mu_3)_w=K_v$, 
and hence $\mu_3\subset K_v$.
On the other hand, $\sigma$ is nontrivial in $\Gal(K(\mu_9)/K)$, 
so $v$ does not split completely in $K(\mu_9)/K$.  
If $K_v$ contained $\mu_9$, then it would contain a $K$-isomorphic copy of $K(\mu_9)$. 
At an unramified place this would force the corresponding local degree in $K(\mu_9)/K$ to be one, or
equivalently force $\Frob_v(K(\mu_9)/K)$ to be the identity.  
This contradicts $\Frob_v(K(\mu_9)/K) = \sigma \neq 1$.  
Thus $\mu_9\not\subset K_v$.
\end{proof}

%%%%%%%%%%%

We now give an example of a number field $K$ which satisfies the assumption of Theorem \ref{thm:numberfield-absolute-trim}  for $p=3$. 

\begin{example}\label{example:p=3_weak_class_non_empty}
Let $\zeta_9$ be a primitive $9$-th root of unity, and let $K \coloneqq \Q(\zeta_9 + \zeta_9^{-1})$ 
be the maximal real subfield of $\Q(\zeta_9)$. 
The cyclotomic degree formula implies $[\Q(\zeta_9): \Q ] = 6$ 
and $[\Q(\zeta_9 + \zeta_9^{-1}) : \Q] = 3$.  
Since $K \subset \R$, we have $\mu_3 \not\subset K$. 
However, both extensions $K(\zeta_3)/K$ and $\Q(\zeta_9) /K$ have degree $2$, and $K(\zeta_3) \subseteq \Q(\zeta_9)$. 
Thus, $K(\mu_3) = K(\zeta_3) = \Q(\zeta_9) = K(\mu_9)$. 
By Lemma \ref{lemma:criterion_p=3}, this shows that $K$ 
satisfies the assumption of Theorem \ref{thm:numberfield-absolute-trim} for $p=3$.  
\end{example}

%%%%%%%%%%%

The class of number fields $K$ with $\mu_3 \not\subset K$ 
which have a place $v$ such that $\mu_3 \subset K_v$ 
but $\mu_9 \not\subset K_v$  
is also non-empty. 

\begin{example}\label{example:p=3_not_weak_class_non_empty}
We have $\mu_3 \not\subset \Q$, $\mu_3 \subset \Q_{\ell}$ but $\mu_9 \not\subset \Q_{\ell}$ 
for $\ell = 7, 13, 31, \ldots$. 
\end{example}

%%%%%%%%%%%%%%%%%

Now we consider the remaining case for $p=3$ 
and show that weak $A_3$-formality fails. 
Example \ref{example:p=3_not_weak_class_non_empty} shows that the following theorem applies 
to a nontrivial class of number fields. 

\begin{theorem}
\label{thm:numberfield-absolute-trim_fails_for_p=3}
Let $K$ be a number field, and assume $\mu_3 \not\subset K$. 
Assume that  $K$ has a place $v$ such that $\mu_3 \subset K_v$ 
but $\mu_9 \not\subset K_v$. 
Then $\Cb(G_K,\F_3)$ is not weakly $A_3$-formal. 
\end{theorem}

\begin{proof}
Let $v$ be a place of $K$ such that $\mu_3 \subset K_v$ 
but $\mu_9 \not\subset K_v$. 
Then the maximal pro-$3$ quotient $G_{K_v}(3)$ of $G_{K_v}$ is a pro-$3$ Demushkin group 
with $q$-invariant $q=3$. 
By \cite[Theorem 6.12]{PQodd}, $\Cb(G_{K_v}(3),\F_3)$ is not $A_3$-formal, 
and hence $\Cb(G_{K_v},\F_3)$ is not $A_3$-formal. 
By Theorem \ref{thm:LG-equivalence}, 
this implies that $\Cb(G_K,\F_3)$ is not weakly $A_3$-formal. 
\end{proof}

%%%%%%%%%%%%

Now we provide a second proof of Theorem \ref{thm:numberfield-absolute-trim} 
which uses a Hasse principle for finite embedding problems of number fields. 
Let $K$ be a number field, 
and let 
$\mathbf E \colon ~ G_K \to A \twoheadleftarrow B$ 
be a finite embedding problem for its absolute Galois group $G_K$. 
We recall that the kernel of $\mathbf E$ 
is defined to be the kernel of $B \twoheadrightarrow A$. 
When $M = \ker(\mathbf E)$ is abelian, conjugation induces an $A$-action on $M$, 
and hence a $G_K$-action through the given homomorphism $G_K \to A$.
We write
$\MD = \Hom(M,\overline K^\times)$ for the Cartier dual. 
For the elementary abelian $p$-kernels below this is simply $\Hom(M,\mu_p)$. 
%

%%%%%%%%%%%%%%

\begin{lemma}\label{lem:cyclotomic-vanishing}
Let $p$ be odd, $\mu_p\not\subset K$, and let $M$ be an elementary
abelian $p$-group whose $G_K$-action has $p$-group image.  
Then $\Sha^1(K,\MD) = 0$.  
\end{lemma}

\begin{proof}
Choose a finite Galois $p$-extension $L_0/K$ killing the action on $M$ and
put $L = L_0K(\mu_p)$.  
The two extensions have relatively prime degrees,
which implies 
$\Gal(L/K)=P\times C$,  
where $P$ is a $p$-group and
$1\ne C\subset \F_p^\times$.  
The group $C$ acts trivially on $M$. 
Indeed, the image of $G_K$ in $\Aut(M)$ is a $p$-group, 
and $C$ has order prime to $p$. 
Hence the restriction of the action to $C$ is trivial. 
Then $C$ acts on $\MD$ by a nontrivial cyclotomic scalar. 
This gives $(\MD)^C=0$ and, since $|C|$ is invertible in $\Fp$, 
$H^j(C,\MD)=0$ for all $j\ge0$.  
The Hochschild--Serre spectral sequence gives
\begin{align}\label{eq:14.2_4.5}
H^1(\Gal(L/K),\MD)=0.
\end{align}
If $\alpha\in\Sha^1(K,\MD)$, then over $L$ the coefficient module is
trivial and
$\res_L(\alpha)$ is a continuous homomorphism $G_L\to \MD$ trivial
on every decomposition group.  
A nonzero such homomorphism would cut out a
nontrivial finite abelian extension of $L$ with a prime having nontrivial
Frobenius, by Chebotarev, 
contradicting local triviality.  
Thus
$\res_L(\alpha)=0$, and inflation--restriction together with \eqref{eq:14.2_4.5}  
gives $\alpha=0$.
\end{proof}

%%%%%%%%%

\begin{theorem}
\label{thm:direct-hasse}
Let $p$ be odd, $\mu_p\not\subset K$, and 
let $\mathbf E$ be a finite embedding problem whose kernel $M$ is elementary abelian of $p$-power order
and whose induced $G_K$-action has $p$-group image. 
If $\mathbf E$ is solvable over every completion of $K$, then it is solvable over $K$.
\end{theorem}

\begin{proof}
By \cite[Theorem 1.1 and Proposition 2.7]{PS}, 
the only obstruction to the lifting of the solutions of the local embedding problems is $\Sha^1(K,\MD)$, 
which vanishes by Lemma \ref{lem:cyclotomic-vanishing}.  
\end{proof}

Now we can prove Theorem \ref{thm:numberfield-absolute-trim} again 
using Theorem \ref{thm:direct-hasse}. 

%%%%%%%%%%%%%

\begin{proof}[Second proof of Theorem \ref{thm:numberfield-absolute-trim}]
Let $V \subset H^1(G_K,\F_p)$ be a finite-dimensional vector subspace. 
For every place $v$ of $K$, the local $A_3$-formality established in the first proof, 
together with Lemma \ref{lem:local-plans-general}, shows that the localised embedding problem 
$\mathbf E^{(3)}_{V,v}$ is solvable. 
The kernel of $\UA_3(V) \twoheadrightarrow V^\vee$ 
is elementary abelian by Remark \ref{rem:kernel_is_elementary_abelian}, 
and its induced $G_K$-action factors through the finite $p$-group $V^\vee$.  
Theorem \ref{thm:direct-hasse}, applied to $\mathbf E^{(3)}_V$, 
therefore gives a global solution.  
Since $V$ was arbitrary, Proposition \ref{prop:two-step-comparison_finite_embedding} 
implies that $\Cb(G_{K},\Fp)$ is weakly $A_3$-formal. 
\end{proof}

%%%%%%%%%%%

\begin{remark}\label{rem:failure_at_p=2_weak_A_3}
We now explain why our argument to determine weak $A_3$-formality of $\Cb(G_K,\F_p)$ does not apply when $p=2$. 
The first proof of Theorem \ref{thm:numberfield-absolute-trim} uses the local-global principle of Theorem \ref{thm:LG-equivalence}. 
As we explained in Remark \ref{rem:failure_at_p=2_local-global}, 
our proof of Proposition \ref{prop:local-to-global} does not apply when $p=2$ for several reasons. 
The second proof of Theorem \ref{thm:numberfield-absolute-trim} uses the local-global principle for the existence 
of solutions of embedding problems for number fields of \cite{PS}. 
The key point is then that $\Sha^1(K,\MD)$ is trivial. 
However, the proof of Lemma \ref{lem:cyclotomic-vanishing} does not apply when $p=2$ 
since $\mu_2 \subset K$. % which implies that 
Hence the vanishing argument used in Lemma \ref{lem:cyclotomic-vanishing}, 
and therefore this proof of the abelian local-global principle, does not extend to $p=2$. 
\end{remark}

%%%%%%%%%%%%%%%%
%
%%%%%%%%%%%%%%%%

\section{Failure of genuine \texorpdfstring{$A_3$}{A3}-formality for number fields}\label{sec:nonformality}

We now prove that, under the cyclotomic hypothesis $\mu_p\not\subset K$, 
the Galois cochain algebra is never genuinely $A_3$-formal. 
The key arithmetic idea is a very particular prescribed ramification theorem for cyclic degree-$p$ extensions, 
due to a result of Gras \cite[Chapter V, Corollary 2.4.2]{GrasCFT}. 
Our use of this result is inspired by the work of Maire--Min\'a\v c--Ramakrishna--T\^an in \cite{MMRT}.  
However, as explained in the introduction, we use the theorem in a different way.  

\begin{notn}
Throughout this section $p$ is odd and $\mu_p\not\subset K$.  
We write $K' \coloneqq K(\mu_p)$, and $\Delta \coloneqq \Gal(K'/K)$. 
The mod-$p$ cyclotomic character embeds $\Delta$ as a nontrivial subgroup
of $\F_p^\times$.  
By a \emph{$p$-extension} of $K$ we mean a
finite Galois extension whose Galois group is a $p$-group.
\end{notn}

%
%%%%%%%%%%%%
%
%%%%%%%%%
%
We recall the following terminology. 

\begin{definition}
\label{def:governing-field}
Let $T$ be a finite set of finite primes of $K$. 
Set 
\[
\mathcal V^T
\coloneqq 
\left\{ x\in K^\times: \operatorname{ord}_{\qq}(x) \equiv0 \pmod p  \text{ for every } \qq \notin T  \right\}.
\]
The
\emph{governing field} attached to $(K,T)$ is defined by adjoining to $K(\mu_p)$ 
the $p$th roots of representatives of
the image of $\mathcal V^T$ in $K'^\times/(K'^\times)^p$, i.e., 
\[
\operatorname{Gov}_{K,T} \coloneqq K'\bigl(\sqrt[p]{\mathcal V^T}\bigr).
\]
\end{definition}

%%%%%%%%%%%%%%

\begin{remark}
The group $\mathcal V^T$ is the group denoted $V^T$ in \cite[Section 2.1]{MMRT}.  
Since the notation conflicts with
our notation $V\subset H^1(G_K,\F_p)$, 
we use the symbol $\mathcal V^T$. 
The quotient of $\mathcal V^T$ relevant to Kummer theory is finite: it is
controlled by the $T$-unit group and the $p$-torsion in the ideal class
group.  Thus $\operatorname{Gov}_{K,T}/K'$ is a finite extension.
\end{remark}

%%%%%%%%%%%%%%%%%%

\begin{lemma}
\label{lem:governing-structure}
The extension $\operatorname{Gov}_{K,T}/K'$ is finite elementary
abelian of exponent $p$, and $\operatorname{Gov}_{K,T}/K$ is Galois.
Moreover, 
\begin{enumerate}[label=\textnormal{(\roman*)}]
\item $\operatorname{Gov}_{K,T}/K'$ is unramified outside the primes of
$K'$ above $T$ and $p$,
\item conjugation by $\Delta$ on
$\Gal(\operatorname{Gov}_{K,T}/K')$ acts via multiplication by a scalar through the
cyclotomic character, and in particular has no nonzero trivial subquotient.
\end{enumerate}
\end{lemma}

\begin{proof}
These facts are listed immediately after the definition of
$\operatorname{Gov}_{K,T}$ in \cite[Section 2.1]{MMRT}. 
Assertion (i) 
follows directly from the valuation condition defining $\mathcal V^T$.   
Assertion (ii) follows from \cite[Chapter I, Theorem 6.2]{GrasCFT}. 
\end{proof}

%%%%%%%%%%%%

%
Let $q\notin T$ be a finite prime with $q\nmid p$.  
Choose a prime $\mathfrak q'$ of $K'$ above $q$.  
The extension $\operatorname{Gov}_{K,T}/K'$ is unramified at $\mathfrak q'$, so it has
a Frobenius element
\[
\sigma_{q,\mathfrak q'}
=\operatorname{Frob}_{\mathfrak q'}
\in \Gal(\operatorname{Gov}_{K,T}/K').
\]
Changing $\mathfrak q'$ changes this element by multiplication by a unit in $\F_p^\times$, 
see \cite[Remark 2.1]{MMRT}.  
We therefore just write $\sigma_q$. 
We recall the following theorem from \cite[Chapter V, Corollary 2.4.2]{GrasCFT}  
(see also \cite[Theorem 1.1]{HMRGras} and \cite[Theorem 2.1]{MMRT}). 

\begin{theorem}[Gras]
\label{thm:gras-condition} 
Let $T$ be a finite set of finite primes of $K$, and let
$S=\{\qq_1,\dots,\qq_s\}$ 
be a nonempty finite set of primes disjoint from $T$ with norm $N\qq_i \equiv 1 \pmod p$ for every $i$.  
There exists a cyclic extension $L/K$ of degree $p$ which is
ramified at exactly the primes of $S$ and in which every prime of $T$
splits completely if and only if 
there exist $a_i \in \F_p^\times$, $i = 1, \ldots, s$, 
such that 
\begin{align}\label{eq:gras-condition}
\sum_{i=1}^s a_i\sigma_{q_i}=0
~ \text{in} ~ \Gal(\operatorname{Gov}_{K,T}/K'). 
\end{align}
\end{theorem}

%%%%%%%%%%%%%%%%%%

\begin{remark}
The important point in \eqref{eq:gras-condition} is not merely linear dependence but 
every coefficient $a_i$ must be nonzero.  
A relation supported on a proper subset of $S$ corresponds to
allowing the cyclic extension to be ramified on that smaller subset.
Thus, for our purposes, 
the theorem converts the problem  
of constructing a cyclic $p$-extension ramified exactly at $S$ and split at $T$ 
into a finite-dimensional linear condition on Frobenius elements in one
fixed Kummer extension.  
A new proof generalising the 
tame Gras--Munnier theorem is given by Hajir--Maire--Ramakrishna 
in \cite[Theorem 1.1]{HMRGras}, 
which also provides the formulation of the criterion we use 
in Theorem \ref{thm:gras-condition} (see also \cite[Theorem 2.1]{MMRT}).  
\end{remark}

For our argument below, $S$ consists of a single prime.  

\begin{corollary}
\label{cor:gras-one-prime}
Let $v\notin T$ be a finite prime with $v\nmid p$.  
Suppose that $v$ splits completely in
$\Gov_{K,T}/K$.  
Then there exists a cyclic extension
$L/K$ of degree $p$ which is ramified exactly at $v$ and in which every
prime of $T$ splits completely. 
\end{corollary}

\begin{proof}
For $S=\{v\}$, condition \eqref{eq:gras-condition} is
$a\sigma_v=0$ with $a \in \F_p^\times$, 
hence simply $\sigma_v=0$, 
i.e., $v$ splits completely in
$\Gov_{K,T}/K$. 
This implies $Nv \equiv 1 \pmod p$ 
since $\Gov_{K,T}$ contains $K'=K(\mu_p)$. 
Thus, all hypotheses of Theorem  \ref{thm:gras-condition} are satisfied, 
and the assertion is the singleton case of the theorem. 
\end{proof}

%%%%%%%%%%%%%%%%%%

For the application of Corollary \ref{cor:gras-one-prime} in Proposition \ref{prop:fresh-ramified-character}, 
we need the following lemma. 

\begin{lemma}
\label{lem:governing-disjointness}
Let $F/K$ be a finite Galois $p$-extension.  
Then $FK'$ and $\operatorname{Gov}_{K,T}$ are linearly disjoint over $K'$, i.e., 
$FK'\cap\operatorname{Gov}_{K,T}=K'$. 
\end{lemma}

\begin{proof}  
This is \cite[Lemma 2.2]{MMRT}.  
We give the argument since it explains where $\mu_p\not\subset K$ is used.
The degrees $[F:K]$ and $[K':K]$ are relatively prime, so
$F\cap K'=K$.  Since both extensions are Galois,
\[
\Gal(FK'/K) \cong \Gal(F/K)\times\Delta.
\]
Therefore, $\Delta$ acts trivially by conjugation on
$\Gal(FK'/K')$ and on all of its quotients.
Let 
$I \coloneqq FK'\cap\operatorname{Gov}_{K,T}$. 
If $I\ne K'$, then $\Gal(I/K')$ is a nonzero elementary
abelian $p$-group.  
As a quotient of
$\Gal(FK'/K')$, it has trivial $\Delta$-action.  
As a quotient of $\Gal(\operatorname{Gov}_{K,T}/K')$, it has
cyclotomic scalar $\Delta$-action by
Lemma~\ref{lem:governing-structure}.  
Since $\mu_p\not\subset K$, this scalar character is nontrivial.  
A nonzero $\F_p$-space cannot
simultaneously carry the trivial action and a nontrivial scalar action.
Hence $I=K'$.
\end{proof}

%%%%%%%%%%%%%%%%%%%%%%%%

\begin{proposition}
\label{prop:fresh-ramified-character}
Let $T$ be a finite set of finite primes containing all primes above $p$,
let $F/K$ be a finite Galois $p$-extension, and let
$g \in \Gal(F/K)$.  
Then there exist a finite prime
$v\notin T$ and a class 
$c\in H^1(G_K,\F_p)$ 
such that:
\begin{enumerate}[label=\textnormal{(\roman*)}]
\item $v\nmid p$ and $\mu_p\subset K_v$,
\item the Frobenius conjugacy class of $v$ in $F/K$ is the conjugacy
class of $g$,
\item the cyclic degree-$p$ extension determined by $c$ is ramified exactly
at $v$,
\item $\operatorname{loc}_{q}(c)=0$ for every $q \in T$. 
\end{enumerate}
If $F_0\subset F$ is a Galois subextension on which $g$ restricts to the
identity, then $v$ splits completely in $F_0/K$.
\end{proposition}

\begin{proof}
Since $F\cap K'=K$, there is a unique
$\widetilde g \in \Gal(FK'/K')$ 
whose restriction to $F$ is $g$.  
By Lemma \ref{lem:governing-disjointness}, restriction gives
\[
\Gal \bigl(FK'\operatorname{Gov}_{K,T}/K'\bigr)
\cong
\Gal(FK'/K')
\times
\Gal(\operatorname{Gov}_{K,T}/K').
\]
Hence the pair $(\widetilde g,1)$ defines an element $\gamma$ of this
Galois group.  
We regard $\gamma$ as an element of the Galois group of the
same compositum over $K$.  
By Chebotarev's density theorem (see for example \cite[Theorem 9.1.3]{NSW}), 
there are infinitely many primes outside $T$ and above residue characteristic
prime to $p$ whose Frobenius conjugacy class is the conjugacy class of
$\gamma$.   
We choose one and call it $v$.
%
%%%
%
The restriction of $\gamma$ to $K'$ is trivial, so $v$ splits completely
in $K'/K$ and therefore $\mu_p\subset K_v$.  Its restriction to $F$ is
conjugate to $g$, proving (ii).  
Its restriction to $\Gov_{K,T}$ is trivial, and its restriction to $K'$ is trivial. 
Hence $v$ splits completely in $\Gov_{K,T}/K$.   
Thus, $v$ satisfies the assumption of Corollary \ref{cor:gras-one-prime}, 
which gives a cyclic extension $L/K$ of degree $p$, ramified
exactly at $v$, in which every prime of $T$ splits completely.  
The associated character
$c \colon G_K \to \Gal(L/K) \cong \Z/p\Z$   
has properties (iii) and (iv).
Finally, if $g$ restricts trivially to $F_0$, then every conjugate of $g$
does too, so the Frobenius of $v$ in $F_0/K$ is trivial.  Hence $v$
splits completely in $F_0$.
\end{proof}

%%%%%%%%%%%%%%%%%%%%%%%%%%%

The next lemma shows that we can find a cohomology class which is orthogonal to a chosen subspace. 

\begin{lemma}
\label{lem:fresh-cup-orthogonality}
Let $V\subset H^1(G_K,\F_p)$ be a finite-dimensional vector subspace, 
and let $T$ contain the primes above $p$ and every finite prime at which some element
of $V$ is ramified.  
In Proposition \ref{prop:fresh-ramified-character}, suppose
that $g \in \Gal(F/K)$ 
restricts trivially to the elementary abelian $p$-extension determined by $V$.  
Then
$c\cup V = 0$.
\end{lemma}

\begin{proof}
The final assertion of Proposition \ref{prop:fresh-ramified-character} implies
that $v$ splits completely in the elementary extension cut out by $V$. 
Thus, 
$\operatorname{loc}_v(V)=0$. 
Fix $a\in V$.  We show that every localisation of $a \cup c$ is zero.  
At $v$, this follows from $a_v=0$.  
At a prime $q \in T$, it follows from $c_{q}=0$.  
At every remaining finite prime, both $a$ and $c$ are unramified.  
Their cup product is zero because both classes
inflate from the maximal unramified pro-$p$ quotient, which has $p$-cohomological dimension one.  
Since $p$ is odd, archimedean places contribute nothing.
Thus, $a \cup c$ lies in $\Sha^2(K,\F_p)$.  
Poitou--Tate duality identifies this group with the dual of $\Sha^1(K,\mu_p)$.  
The latter vanishes by Lemma~\ref{lem:cyclotomic-vanishing}, applied with the trivial
module $M=\F_p$.  
Therefore, we get  $a\cup c=0$ globally. 
Since $p$ is odd, graded commutativity implies $c \cup a = 0$.  
Hence $c \cup V = 0$. 
\end{proof}

%%%%%%%%%%%%%%%%%

%
Now let $V\subset H^1(G_K,\F_p)$ be a finite-dimensional vector subspace, 
and let $\mathbf E^{(3)}_V \colon   G_K \xrightarrow{\chi_V}  V^\vee \twoheadleftarrow \UA_3(V)$ 
be the associated embedding problem. % studied in Section \ref{sec:finite-problem}. 
For a solution
$\rho \colon G_K \to \UA_3(V)$, 
we write 
$\brr \colon G_K \to \UA_2(V)$ 
for the composition with the canonical surjection 
$\UA_3(V) \twoheadrightarrow \UA_2(V)$ 
and call $\brr$ the \emph{quadratic quotient} of $\rho$.  
By Lemma \ref{lemma:quadratic_lift_exists}, 
there is a cochain $b_{\rho} \in \Ch^1(G_K, R_V^\vee)$ 
such that  
\[
\brr(\sigma)
= 1 + \chi_V(\sigma) + b_\rho(\sigma). 
\]
\begin{remark}\label{rem:b_rho_hom}
We note that $b_{\rho}$ restricts to a homomorphism on $\ker\chi_V$ 
since $\delta b_{\rho} = - \chi_V \cup \chi_V$ and hence $\delta b_{\rho}|_{\ker\chi_V} = 0$. 
Moreover, for $\sigma \in \ker(\chi_V)$, 
we have $\brr (\sigma) = 1 + b_\rho(\sigma)$. 
Thus, $\brr|_{\ker(\chi_V)} = 1$ if and only if $b_\rho |_{\ker(\chi_V)} = 0$.   
\end{remark}
We now show  that a solution whose quadratic quotient is nontrivial on $\ker \chi_V$ 
cannot extend through every finite-dimensional enlargement. 

\begin{proposition}
\label{prop:killing-lemma}
Let $V\subset H^1(G_K,\F_p)$ be a finite-dimensional vector subspace and let
$\rho \in\Sol(\mathbf E^{(3)}_V)$.  
If $\brr|_{\ker\chi_V}\ne1$, 
then there is a class $c \in H^1(G_K,\F_p)$ such that 
$V \subsetneq W \coloneqq V + \F_p c$ 
and 
$\rho \notin \im  (\Sol(\mathbf E^{(3)}_W) \to \Sol(\mathbf E^{(3)}_V))$. 
\end{proposition}

\begin{proof}
Let $b_{\rho} \in \Ch^1(G_K, R_V^\vee)$ be the quadratic coordinate of $\rho$. 
Since we assume $\brr|_{\ker\chi_V} \ne 1$, 
we have $b_{\rho}|_{\ker\chi_V} \ne 0$ by Remark \ref{rem:b_rho_hom}. 
Hence we can choose $\sigma \in \ker \chi_V$ with $b_\rho(\sigma)\ne0$. 
We then choose 
$r \in R_V$ such that  $b_\rho(\sigma)(r) \ne 0$.   
Let $F/K$ be the finite Galois $p$-extension such that $\Gal(F/K) = \im \rho$. 
We identify $\Gal(F/K)$ with the image of $\rho$,  and put $g=\rho(\sigma)$.  
Let $T$ contain the primes above $p$ and all finite primes at which some element of $V$ is ramified.  
%
%%%%
%
Now we apply Proposition \ref{prop:fresh-ramified-character} to $(F,g,T)$.
Let $v\notin T$ and $c\in H^1(G_K,\F_p)$ be the resulting finite prime and class, respectively. 
The image of $g$ in $V^\vee$ is $\chi_V(\sigma)=0$. 
Therefore, $g$ restricts  trivially to the elementary abelian extension determined by $V$, 
and the chosen prime $v$ splits completely in that extension.  
Hence $\loc_v(V)=0$ and $c \cup V = 0$ 
by Lemma \ref{lem:fresh-cup-orthogonality}.  
Moreover, $c \notin V$, since $\loc_v(V) = 0$ whereas $c_v$ is ramified and nonzero. 
%
%%%
%
The prime $v$ is unramified in $F/K$, and its Frobenius in
$\Gal(F/K)$ is conjugate to $g$. 
Since $v$ splits completely in the extension determined by $V$, 
we have $\chi_V|_{G_{K_v}} = 0$. 
By Remark \ref{rem:b_rho_hom}, 
this implies that  
the restriction $b_\rho(r)|_{G_{K_v}}$ is a $1$-cocycle in $\Ch^1(G_{K_v},\F_p)$.  
To simplify the notation, we write 
$d_v \coloneqq b_\rho(r)|_{G_{K_v}} \in H^1(G_{K_v},\F_p)$.   
Since $v$ is unramified in $F/K$, 
the character $d_v$ is unramified. 
The element $g = \rho(\sigma)$ has zero degree one coordinate, and conjugating it by an element of $\UA_3(V)$ 
changes its degree two coordinate only by commutators of total degree at least three. 
Thus, the image of a Frobenius element at $v$ has quadratic coordinate $b_\rho(\sigma)$,  
and hence $d_v(\Frob_v) = b_\rho(\sigma)(r) \neq 0$.
Thus, $d_v$ is a nonzero unramified character.
%
%%%%%%%%%
%
By (i) in Proposition \ref{prop:fresh-ramified-character}, 
we have $v\nmid p$ and $\mu_p\subset K_v$.  
By local Tate duality, the cup-product is non-degenerate and makes
$H^1(G_{K_v},\F_p)$ a two-dimensional symplectic $\F_p$-space.
Its unramified line is isotropic, hence equals its own orthogonal
complement.  
The class $d_v$ is a nonzero unramified class, whereas $c_v$
is ramified and therefore does not lie on the unramified line.  
This shows $d_v \cup c_v \ne 0$.  

%
%%%%%%%%%%%%%%

Now we set $W \coloneqq V+\F_p c$  and write
$r = \sum_i s_i\otimes t_i$ for $s_i, t_i\in V$. 
Since $c \cup V=0$, 
each $t_i\otimes c$ belongs to $R_W$ by graded-commutativity as $p$ is odd, 
and hence $T_r \coloneqq r\otimes c = \sum_i s_i\otimes t_i\otimes c \in W \otimes R_W$. 
Since $r \in R_V \subset R_W$, 
we have $r \otimes c \in R_W \otimes W$. 
This shows 
$T_r \in (R_W\otimes W)\cap(W\otimes R_W) = K_3(W)$. 
Now suppose that there is a solution $\rho_W$ of $\mathbf E^{(3)}_W$ 
which maps to $\rho$.  
We write $\rho_W = 1+\chi_W+b_{2,W}+b_{3,W}$ 
with $b_{i,W} \in \Ch^1(G_K,P_i(W))$ for $i=2,3$.  
By Lemma \ref{lemma:nic_lift_cochain}, 
we have $\delta b_{3,W} = \omega(b_{2,W}) \coloneqq -\chi_W \cup b_{2,W} - b_{2,W} \cup \chi_W \in \Zh^2(G_K,P_3(W))$.  
Evaluating this identity on $T_r\in K_3(W)$ and taking cohomology gives 
\[ 
[\omega(b_{2,W}) (T_r)] = 0 ~ \text{in} ~ H^2(G_K,\F_p). 
\]
Hence we have $\loc_v([\omega(b_{2,W}) (T_r)]) = 0 $ in $H^2(G_{K_v},\F_p)$ as well. 
Now we compute the class $\loc_v([\omega(b_{2,W}) (T_r)])$ in a different way. 
Since $T_r = \sum_i s_i\otimes t_i\otimes c \in W \otimes R_W$,  
we get 
$(\chi_W \cup b_{2,W} )(T_r) = \sum_i s_i \cup b_{2,W}(t_i\otimes c)$. 
Since $s_i \in V$ and $\loc_v(V)=0$, 
the localisation of $(\chi_W \cup b_{2,W} )(T_r)$ at $v$ is zero.  
As $T_r = \sum_i s_i\otimes t_i\otimes c \in R_W \otimes W$,   
we get 
$(b_{2,W} \cup \chi_W )(T_r) = b_{2,W}(r)  \cup c$. 
Since $\rho_W$ maps to $\rho$ by assumption, 
the restriction of $b_{2,W}$ to $R_V$ is $b_\rho$. 
Thus, we get 
$\loc_v([\omega(b_{2,W}) (T_r)]) = - b_\rho(r)|_{G_{K_v}} \cup c_v 
= - d_v \cup c_v$.    
Thus, the existence of $\rho_W$ would force the cohomology class  $d_v \cup c_v$ to vanish.  
This contradicts our previous computation of $d_v \cup c_v \ne 0$. 
Hence $\rho$ does not lift to a solution of $\mathbf E^{(3)}_W$. 
\end{proof}

%%%%%%%%%%

Proposition \ref{prop:killing-lemma}  shows that a solution which lifts through every finite 
enlargement must have a special quadratic quotient.  
The next two lemmas rule out that remaining possibility after a suitable  two-dimensional choice of $V$.

%%%%%%%%%%%%%%%%%

\begin{lemma}
\label{lem:isotropic-pair}
There exist linearly independent
$a,b\in H^1(G_K,\F_p)$  
with $a \cup b=0$.
\end{lemma}

\begin{proof}
We apply Proposition \ref{prop:fresh-ramified-character} first with the trivial
$p$-extension $F=K$, $g=1$, and $T$ equal to the set of primes above $p$.  
This gives a nonzero character $a$, ramified at one tame prime.  
We enlarge  $T$ to contain the primes above $p$ and every prime at which $a$ is ramified. 
Now we apply Proposition \ref{prop:fresh-ramified-character} again with
$F=K(a)$ and $g=1$.  
The resulting character $b$ is ramified at a new
prime $v$ which splits completely in $K(a)$, and it is trivial at the
primes of $T$.  
Now we apply Lemma \ref{lem:fresh-cup-orthogonality} to $V = \F_p a$ 
and get $b \cup a = 0$ and hence $a \cup b = 0$ by graded commutativity as $p$ is odd.  
The classes are linearly independent since $b$ is ramified at $v$ while $a_v = 0$.
\end{proof}

%%%%%%%%%%%%%%%%%

\begin{lemma}
\label{lem:quadratic-nonsplit}
Suppose $V$ contains linearly independent classes $a,b$ such that
$a \cup b=0$.  
Then the natural surjection
$\UA_2(V)\twoheadrightarrow V^\vee$  
has no group-theoretic section.
\end{lemma}

\begin{proof}
Since $a,b$ are linearly independent, 
we can choose $x,y\in V^\vee$ with
$x(a) = y(b) = 1$, 
and 
$x(b) = y(a) = 0$. 
Since $a \cup b=0$, 
the tensor $a\otimes b$ belongs to $R_V$. 
Let $\widetilde x = 1 + x+x_2$ and $\widetilde y = 1+y+y_2$ be arbitrary lifts to $\UA_2(V)$ 
with $x_2,y_2 \in P_2(V)$. 
Since the subgroup $P_2(V)$ is central in $\UA_2(V)$, the degree two coordinate of the commutator 
$[\widetilde x,\widetilde y]$ is $xy-yx$. %\in P_2(V)$. 
Evaluating this coordinate on $a\otimes b$ gives
$ (xy-yx)(a\otimes b)  = x(a)y(b)-y(a)x(b) = 1$.  
Thus, $[\widetilde x,\widetilde y] \ne 1$.   
Since $V^\vee$ is abelian, this shows that there cannot exist a section $V^\vee \to \UA_2(V)$. 
\end{proof}

%%%%%%%%%%%%%%%%%

Now we are ready to prove our main result. 

\begin{theorem}
\label{thm:numberfield-nonformal}
Let $K$ be a number field, let $p$ be an odd prime number, and assume
$\mu_p\not\subset K$.  
Then
$\Cb(G_K,\F_p)$ is not $A_3$-formal.  
\end{theorem}

%%%%%%%%%%%%%%%

\begin{proof}
By Lemma \ref{lem:isotropic-pair}, 
we can choose linearly independent $a,b\in H^1(G_K,\F_p)$ with $a \cup b=0$,  
and we set 
$V \coloneqq  \F_p a + \F_p b$.  
Suppose that $\Cb(G_K,\F_p)$ were $A_3$-formal. 
By Theorem  \ref{thm:compatible-defining-systems}, 
there would be a compatible family of solutions 
\[
(\rho_W)_{W\in\mathcal F(H^1(G_K,\F_p))}
\in
\varprojlim_W\Sol(\mathbf E^{(3)}_W). 
\]
Let $\rho = \rho_V$ and let $\brr \colon G_K \to \UA_2(V)$ be its quadratic quotient.  
%
%%%%%
%
If $\brr|_{\ker\chi_V}=1$, then $\brr$ factors uniquely through 
$G_K/\ker(\chi_V)$. 
Since $\chi_V$ is surjective by Lemma \ref{lemma:chi_V_continuous_surjective}, 
$\chi_V$ induces an isomorphism  $G_K/\ker(\chi_V) \cong V^\vee$. 
Thus, there exists a group homomorphism
$s \colon V^\vee \to \UA_2(V)$ 
with $s \circ \chi_V = \brr$. 
Composition with the quotient map $q \colon \UA_2(V) \to V^\vee$ 
gives  $q \circ s \circ \chi_V = q \circ \brr = \chi_V$. 
Since $\chi_V$ is surjective by Lemma \ref{lemma:chi_V_continuous_surjective}, 
this implies $q \circ s  = \id_{V^\vee}$. 
This contradicts Lemma \ref{lem:quadratic-nonsplit}.  
Therefore, we must have 
$\brr|_{\ker\chi_V} \ne 1$. 
Proposition \ref{prop:killing-lemma} gives a finite-dimensional
$W \subset H^1(G_K,\F_p)$ with $V \subset W$ 
such that $\rho$ is not in the image of
$\Sol(\mathbf E^{(3)}_W)\longrightarrow\Sol(\mathbf E^{(3)}_V)$.  
But the $W$-component $\rho_W$ of the projective family maps to
$\rho_V = \rho$, 
contradicting the choice of $W$.  
Thus, $\Cb(G_K,\F_p)$ is not $A_3$-formal.
\end{proof}

%%%%%%%%%%%%%%

\end{document}